\documentclass[hidelinks,12pt]{article}

\usepackage[margin=1in]{geometry}  

\usepackage{BOONDOX-cal}
\usepackage{amsmath}               
\usepackage{amsfonts}              
\usepackage{amsthm}                
\usepackage{amssymb}
\usepackage{titlesec}
\usepackage{verbatim}
\usepackage{hyperref}
\usepackage{theoremref}

\usepackage{bm}

\newtheoremstyle{break}
  {\topsep}{\topsep}%
  {\itshape}{}%
  {\bfseries}{}%
  {\newline}{}%
\theoremstyle{break}
\newtheorem{thm}{Theorem}

\theoremstyle{definition}
\newtheorem{lemma}[thm]{Lemma}

\newtheorem{prop}[thm]{Proposition}

\theoremstyle{definition}
\newtheorem{corollary}[thm]{Corollary}

\theoremstyle{definition}
\newtheorem{remark}{Remark}

\newtheorem{mydef}{Definition}

\newcommand{\RR}{\mathbb{R}}      
\newcommand{\ZZ}{\mathbb{Z}}      
\newcommand{\CC}{\mathbb{C}}
\newcommand{\D}{\text{d}}
\newcommand{\IFF}{if and only if }
\usepackage{physics}

\numberwithin{equation}{section}

\begin{document}

\title{\textbf{On the commutation properties of finite convolution and differential operators II: sesquicommutation}}

\author{Yury Grabovsky, \qquad Narek Hovsepyan}

\date{}
\maketitle

\begin{abstract}
We introduce and fully analyze a new commutation relation $\overline{K} L_1 =
L_2 K$ between finite convolution integral operator $K$ and differential
operators $L_1$ and $L_{2}$, that has implications for spectral properties of
$K$. This work complements our explicit characterization of commuting pairs
$KL=LK$ and provides an exhaustive list of kernels admitting commuting or
sesquicommuting differential operators.
\end{abstract}

\tableofcontents

\section{Introduction}

In many applications the question of understanding spectral properties of finite convolution integral operators

\begin{equation} \label{K}
(Ku)(x) = \int_{-1}^1 k(x-y) u(y) \D y
\end{equation} 
are important. In some cases one is able to find a differential operator $L$
commuting with $K$ (cf. \cite{morrison,widom,katsevich1,gr hov commut}),
\begin{equation} \label{C3}
KL = L K.
\end{equation}

\noindent Then, eigenfunctions of $K$ can be chosen to be solutions of ordinary differential equations. This allows to obtain analytical information about the eigenvalues and eigenfunctions of integral operators, using the vast literature on asymptotic properties of solutions of ordinary differential equations.

An example of this phenomenon is the result of Widom \cite{widom}, where the author obtained asymptotic behavior of the eigenvalues of a family of integral operators with real-valued even kernels, using comparison with special operators that commute with differential operators. A complete characterization of such special operators commuting with symmetric second order differential operators was obtained by Morrison \cite{morrison} (see also \cite{gr hov commut}, \cite{grunbaum}, \cite{wright},). We are interested in the possibility of extending these ideas to complex-valued kernels $k(z)$. In this more general context the property of commutation must also be generalized, so as to permit the characterization of eigenfunctions as solutions of an eigenvalue problem for a second or fourth order differential operator. 

To that end we introduce a new kind of commutation relation, referred to as \textit{sesquicommutation}:  

\begin{equation} \label{C1}
\begin{cases}
\overline{K} L_1 = L_2 K,
\\
L_j^T = L_j, 
\end{cases}\quad j=1,2 ,
\tag{C}
\end{equation}

\noindent where $L_1, L_2$ are differential operators with complex coefficients (we note that the condition $L=L^T$ also includes constraints on the boundary values of coefficients of $L$).  It can be easily checked that in this case

\begin{equation} \label{L1 K* K = conj K* K L1}
L_1 K^*K = \overline{K^* K} L_1 .
\end{equation}

\noindent Let now $\lambda$ be a singular value of $K$ corresponding to singular function $u$, i.e. $K^* K u = \lambda u$, clearly $\lambda \in \RR$ and therefore we find $\lambda \overline{L_1 u} = K^* K \overline{L_1 u}$. It follows that $\overline{L_1 u}$ is either zero, or an eigenfunction of $K^*K$ with the same eigenvalue $\lambda$. If the corresponding eigenspace of $K^*K$ is one-dimensional, then there exists a complex number $\sigma$ such that

\begin{equation*}
L_1 u = \sigma \overline{u} .
\end{equation*}

\noindent Otherwise, applying \eqref{L1 K* K = conj K* K L1} to $\overline{L_1 u}$ we find that

\begin{equation*}
K^* K (L_1^* L_1 u) = \lambda L_1^* L_1 u ,
\end{equation*} 

\noindent hence eigenspaces of $K^* K$ are invariant under the fourth order self-adjoint operator $L_1^* L_1$. In particular, there exists an eigenbasis of $K^* K$ consisting of eigenfunctions of $L_1^* L_1$.  Moreover, transposing the sesquicommutation relation and then taking adjoint we find $K L_1^* = L_2^* \overline{K}$, which along with \eqref{C1} implies

\begin{equation*}
K L_1^* L_1 = L_2^* L_2 K .
\end{equation*}  

\noindent In particular if $L_1 = L_2 =: L$ we see that $L^*L$ commutes with $K$ (and also with $K^*$), hence eigenspaces of $L^*L$ are invariant under $K$ and $K^*$.

Under the assumptions that $k$ is analytic at $0$ and $K$ is self-adjoint we analyze the sesquicommutation relation \eqref{C1}. In Theorem~\ref{THM reduction} we show that if $k$ is nontrivial (see Definition~\ref{trivial DEF}), then either $L_1 = L_2$ or $L_1 = -L_2$. The latter case yields only trivial kernels (cf. Theorem~\ref{THM sesqui comm L2=-L1}). The results in the former case are listed in Theorem~\ref{THM sesqui comm}. Note that Morrison's result lies in the intersection of commutation and sesquicommutation (with $L_1 = L_2$), when $K$ is real and self-adjoint. In fact, in this case sesquicommutation actually reduces to commutation. 

\begin{remark} \label{REM example sesqui}
As a particularly interesting example derived from sesquicommutation, we mention that the eigenfunctions of the compact self-adjoint integral operator $K$ with kernel $\displaystyle k(z) = \frac{e^{-i\frac{\pi}{4}z}}{\cos \frac{\pi}{4}z} + \frac{z e^{i \frac{\pi}{4}z}}{\sin \frac{\pi}{2}z}$ are eigenfunctions of the fourth order self-adjoint differential operator $L^* L$, where

\begin{equation*}
L = - \tfrac{\D}{\D y} \left[ \cos \left( \tfrac{\pi y}{2} \right)\tfrac{\D}{\D y} \right]  + \tfrac{\pi^2}{32} e^{i\tfrac{\pi y}{2} }.
\end{equation*}

\noindent Moreover, if eigenspaces of $K$ are one-dimensional, then eigenfunction $u$ of $K$ satisfies a second order differential equation $L u = \sigma \overline{u}$ for some $\sigma \in \CC$.

\end{remark}

\section{Preliminaries}

We assume that $k(z) \in L^2((-2,2), \CC)$ is analytic in a neighborhood of $0$. Further, assume that $L_j$ are second order differential operators:

\begin{equation} \label{L}
\begin{cases}
L u = \mathcal{a} u'' + \mathcal{b} u' + \mathcal{c} u,
\\
\mathcal{a}(\pm 1) = 0, \ \mathcal{b}(\pm 1) = \mathcal{a}'(\pm 1) ,
\end{cases}
\end{equation}

\noindent where the indicated boundary conditions are necessary for the sesquicommutation relation to hold. They are also necessary for symmetry of differential operators, in which case we will only be specifying additional constraints on the coefficients of $L$, always assuming that the boundary conditions in \eqref{L} hold. In particular operators $L_j$ have to be of Sturm-Liouville type, since $L = L^T$ implies that $\mathcal{b} = \mathcal{a}'$. Thus

\begin{equation} \label{L1, L2}
\begin{cases}
L_j u = (\mathcal{b}_j u')' + \mathcal{c}_j u,
\\
\mathcal{b}_j(\pm 1) = 0,
\end{cases} \quad j=1,2.
\end{equation}

\noindent Due to the imposed boundary conditions it is a matter of integration by parts to rewrite \eqref{C1} as

\begin{equation}\label{R1}
\begin{split}
\mathcal{b}_1(y) \overline{k''(z)} - \mathcal{b}_2(y+z) k''(z) - \mathcal{b}_1'(y) \overline{k'(z)} - \mathcal{b}_2'(y+z) k'(z) + &
\\
+\mathcal{c}_1(y) \overline{k(z)} - \mathcal{c}_2(y+z) k(z) &= 0 .
\end{split}
\tag{R}
\end{equation}

The main idea of the proof is to analyze \eqref{R1} by differentiating it w.r.t. $z$ sufficient number of times and evaluating the result at $z=0$. This allows one to find relations between the coefficient functions of the differential operators, and an ODE for the highest order coefficient. Once the form of the highest order coefficient is determined, we consequently find the forms of all the other coefficient functions. It turns out that the coefficient functions satisfy linear ODEs with constant coefficients, and therefore are equal to linear combinations of polynomials multiplied by exponentials. We then substitute these expressions into \eqref{R1} and using the linear independence of functions $y^j e^{y \lambda_l}$, obtain equations for $k$. Then the task becomes to analyze how many of these equations can be satisfied by $k$ and how its form changes from one equation to another.

\begin{remark}
The reason that reduction of \eqref{C1} to $L_1 = \pm L_2$  (see Section~\ref{reduction SECTION}) works, is the self-adjointness assumption on $K$. This induces symmetry in \eqref{R1}. More precisely, \eqref{R1} becomes a relation involving the even and odd parts (and their derivatives) of the function $k(z) e^{\frac{\lambda}{2} z}$. And as a result the relations for even and odd parts separate. We then prove that if $L_1 \neq \pm L_2$, then both even and odd parts of $k$ are determined in a way that $k$ becomes trivial. 
\end{remark}

\section{Main Results}

\begin{mydef} \label{trivial DEF}

We will say that $k$ (or operator $K$) is \textit{trivial}, if it is a finite linear combination of exponentials $e^{\alpha z}$ or has the form  $e^{\alpha z} p(z)$, where $p(z)$ is a polynomial. Note that in this case $K$ is a finite-rank operator.

\end{mydef}

\vspace{.1in}

\noindent Let us assume that

\begin{enumerate}
\item[(A)] $K$ is self-adjoint, so $k(-z) = \overline{k(z)}, \qquad z \in [-2,2]$.

\end{enumerate}

\begin{thm}[Reduction of sesquicommutation] \label{THM reduction}
Let $K, L_1, L_2$ be given by \eqref{K} and \eqref{L1, L2} with $\mathcal{b}_j, \mathcal{c}_j, k$ smooth in $[-2,2]$. Assume $k$ is nontrivial, (A) holds, and $k$ is analytic at $0$, but not identically zero near $0$. Then \eqref{C1} implies either $L_1 = L_2$ or $L_1 = -L_2$.
\end{thm}

\vspace{.1in}

\begin{remark} \label{REM multiplier sesqui}
Let $M$ be the multiplication operator by $z \mapsto e^{\tau z}$ with $\tau \in i\RR$, then $M K M^{-1}$ is a finite convolution operator with kernel $k(z) e^{\tau z}$ (where $k$ is the kernel of $K$), which is also self-adjoint since so is $K$. If $K$ sesquicommutes with $L$, i.e. $\overline{K}L = L K$, then $M K M^{-1}$ sesquicommutes with $M^{-1} L M^{-1}$. With this observation the results of Theorem~\ref{THM sesqui comm} are stated up to multiplication of $k$ by $e^{\tau z}$, i.e. we chose a convenient constant $\tau$ in order to more concisely state the results.
\end{remark}

\begin{thm}[$L_1 = L_2$] \label{THM sesqui comm}
Let $K, L_1, L_2$ be given by \eqref{K} and \eqref{L1, L2}, with $L_1 = L_2$ and let their coefficient functions be $\mathcal{b}$ and $\mathcal{c}$. Let $\mathcal{b}, \mathcal{c}, k$ be smooth in $[-2,2]$. Further, assume $k$ is nontrivial, (A) holds, $k$ is analytic at $0$, but not identically zero near $0$. Then \eqref{C1} implies (all the used parameters are real, unless stated otherwise)

\begin{enumerate}

\item $\displaystyle k(z) = \frac{\gamma \sinh \mu z}{\mu \sinh \gamma z}$.

\begin{equation*}
\begin{cases}
\displaystyle \mathcal{b}(y) =  \tfrac{1}{2\gamma^2} \left[\cosh(2 \gamma y) - \cosh(2 \gamma) \right],
\\[.1in]
\displaystyle \mathcal{c}(y)= (\gamma^2 - \mu^2) \mathcal{b}(y) + c_0 ,
\end{cases}
\end{equation*}

\noindent where $\mu \in \RR \cup i\RR$ and $c_0 \in \CC$.

\item $\displaystyle k(z)= \alpha e^{-i \mu z} + \frac{\sin \mu z}{z}, \ \ \alpha \neq 0$ and

\begin{equation*}
\begin{cases}
\displaystyle \mathcal{b}(y) = y^2-1,
\\[.1in]
\displaystyle \mathcal{c}(y) = i \mu \mathcal{b}'(y) + \mu^2 \mathcal{b}(y) + \tfrac{\mu}{\alpha}.
\end{cases}
\end{equation*}

\item $\displaystyle k(z) = \frac{ \sinh(2\mu_2) \sinh(\mu_1 z) e^{-\frac{i\pi}{4}z} + \sinh(2\mu_1) \sinh(\mu_2 z) e^{\frac{i\pi}{4}z} }{\mu_1 \mu_2 \sin \frac{\pi z}{2}}$ and

\begin{equation} \label{b, c sesq}
\begin{cases}
\displaystyle \mathcal{b}(y) = -\cos  \tfrac{\pi y}{2},
\\[.1in]
\displaystyle \mathcal{c}(y) =  i \tfrac{\mu_2^2 - \mu_1^2}{\pi} \mathcal{b}'(y) - \left( \tfrac{\pi^2}{16} + \tfrac{\mu_1^2 + \mu_2^2}{2} \right) \mathcal{b}(y) ,
\end{cases}
\end{equation}
where $\mu_1, \mu_2 \in \RR \cup i\RR$. In the special case  $\mu_1 =i \mu$; \ $\mu_2 =i( \mu \pm \frac{\pi}{2})$ with $\mu \in \RR$, to $\mathcal{c}(y)$ a complex multiple of $e^{-2i(\frac{\pi}{4} \pm \mu)y}$ can be added.
\end{enumerate}
\end{thm}

\begin{remark} \mbox{}
\begin{enumerate}

\item[(i)] In items 1 and 3, if $\mu, \mu_j$ or $\gamma =0$, one takes appropriate limits. Note that $k$ can be multiplied by arbitrary real constant and  $L_1=L_2$ by a complex one.

\item[(ii)] Using the same proof techniques one can easily check that under the given assumptions of the theorem, no kernel would satisfy the sesquicommutation relation, when $L_1=L_2$ is a first order operator.

\item[(iii)] In item 1, $K$ is real valued and self-adjoint, in particular sesquicommutation reduces to commutation and we recover Morrison's result.

\item[(iv)] Widom's theory of asymptotics of eigenvalues applies only if
  $k(z)$ has an even extension to $\RR$ such that $\hat{k}(\xi)$ is
  nonnegative and monotone decreasing, at least when $\xi \to \infty$. Item 2
  corresponds to $\hat{k}(\xi)$ being a characteristic function of an interval
  plus a delta-function, centered anywhere one likes. Item 3 is the most
  puzzling, it is unknown if there is an extension whose Fourier transform is nonnegative and monotone decreasing. Item 1 are all even kernels.
\end{enumerate}

\end{remark}

\vspace{.2in}

\noindent From the discussion in the introduction we immediately obtain:

\begin{corollary} \label{CORO eigenfunctions}
Let $K$ be one of the operators of Theorem~\ref{THM sesqui comm} and let $L$ be corresponding operator that sesquicommutes with it (i.e. $\overline{K}L = L K$), then $L^*L$ commutes with $K$. In particular, the eigenfunctions of $K$ are eigenfunctions of the fourth order self-adjoint differential operator $L^*L$. Moreover, if eigenspaces of $K$ are one-dimensional, then eigenfunction $u$ of $K$ satisfies second order differential equation $Lu = \sigma \overline{u}$ for some $\sigma \in \CC$.
\end{corollary}

\begin{remark}
The example mentioned in Remark~\ref{REM example sesqui} in the introduction is obtained from item 3 of Theorem~\ref{THM sesqui comm} by choosing $\mu_2 = 0, \ \mu_1 = \frac{i \pi}{4}$.
\end{remark}

\begin{thm}[$L_1 = -L_2$] \label{THM sesqui comm L2=-L1}
Let $K, L_1, L_2$ be given by \eqref{K} and \eqref{L1, L2}, with $L_1 = -L_2$ and let the coefficients of $L_1$ be $\mathcal{b}$ and $\mathcal{c}$. Let $\mathcal{b}, \mathcal{c}, k$ be smooth in $[-2,2]$. Further, assume (A) holds, $k$ is analytic at $0$, but not identically zero near $0$. If \eqref{C1} holds true, then $k$ is trivial.

\end{thm}

\section{Relations for coefficients} \label{Sesqui comm SECTION}

In this section we consider \eqref{C1} with $L_1, L_2$ given by \eqref{L1, L2}. We assume (A) holds, $k$ is analytic at $0$, but not identically zero near $0$ and finally $k$ is not of the form $e^{\alpha z}$. We aim to find the relations that the coefficient functions $\mathcal{b}_j, \mathcal{c}_j$ must satisfy. Write $k(z) = \sum_{n=0}^\infty \frac{k_n}{n!} z^n$ near $z=0$. The $n$-th derivative of \eqref{R1} w.r.t. $z$ at $z=0$ gives

\begin{equation} \label{nth derivative of sesqui rel}
(-1)^n [\mathcal{b}_1 k_{n+2} +\mathcal{b}_1' k_{n+1} +\mathcal{c}_1 k_n] - \sum_{j=0}^n C_j^n \mathcal{b}_2^{(n-j)} k_{j+2} - \sum_{j=0}^n C_j^n \mathcal{b}_2^{(n-j+1)} k_{j+1} - \sum_{j=0}^n C_j^n \mathcal{c}_2^{(n-j)} k_{j} = 0 ,
\end{equation}

\noindent where $C_j^n = {n \choose j}$, when $n=0$ we get

\begin{equation*}
\begin{split}
k_1 (\mathcal{b}_1' - \mathcal{b}_2') + k_2 (\mathcal{b}_1 - \mathcal{b}_2) + k_0 (\mathcal{c}_1 - \mathcal{c}_2)=0 .
\end{split}
\end{equation*}

\noindent $\bullet$ If $k_0=k_1=0$, then let us show that $k$ is trivial. Assume first $\mathcal{b}_1 \neq \pm \mathcal{b}_2$, then clearly $k_2 = 0$. Let us prove by induction that all $k_j = 0$, which contradicts to the assumption that $k$ doesn't vanish near $0$. Assume $k_j = 0$ for $j=0,...,m$, then \eqref{nth derivative of sesqui rel} for $n=m-1$ reads $\left[ (-1)^{m-1} \mathcal{b}_1 - \mathcal{b}_2 \right] k_{m+1} = 0$, therefore $k_{m+1} = 0$. Let now $\mathcal{b}_1 = \mathcal{b}_2$, assume for the induction step that $k_j = 0$ for $j=0,...,n$, then \eqref{nth derivative of sesqui rel} reads

\begin{equation*}
\left[ (-1)^n - 1\right] k_{n+2} \mathcal{b}_1 + \left[ (-1)^n  - n-1 \right] k_{n+1} \mathcal{b}_1' = 0 .
\end{equation*}

\noindent When $n$ is odd we immediately obtain $k_{n+1} = 0$. When $n$ is even we get $(n+2) k_{n+1} \mathcal{b}_1' + 2 k_{n+2} \mathcal{b}_1 = 0$ and because of boundary conditions $\mathcal{b}_1(\pm 1) = 0$ we deduce $k_{n+1} = k_{n+2} =0$. Finally, the case $\mathcal{b}_1 = -\mathcal{b}_2$ can be done analogously.

\vspace{.1in}

\noindent $\bullet$ If $k_0 = 0, k_1 \neq 0$, by rescaling let $k_1 = 1$ and by considering $e^{-\frac{k_2}{2} z} k(z)$ instead of $k(z)$ (see Remark~\ref{REM multiplier sesqui}) we may assume $k_2 = 0$. Now, $\mathcal{b}_2(y) = \mathcal{b}_1(y) + \alpha$ for some $\alpha \in \CC$. From \eqref{nth derivative of sesqui rel} with $n=1$ we find $\mathcal{c}_2 = - \mathcal{b}_1'' - 2k_3 \mathcal{b}_1 - \mathcal{c}_1 - k_3 \alpha$. Using the obtained expressions, from the relation corresponding to $n=2$ we get 

\begin{equation} \label{c1p}
\mathcal{c}_1' = - \tfrac{1}{2} \mathcal{b}_1''' - k_3 \mathcal{b}_1' + \tfrac{k_4 \alpha}{2} .
\end{equation}

\noindent Now, \eqref{nth derivative of sesqui rel} with $n=3$ reads

$$2\mathcal{b}_1^{(4)} + k_3 \mathcal{b}_1'' -5k_4 \mathcal{b}_1' + 2 (k_3^2 - k_5) \mathcal{b}_1 + 3\mathcal{c}_1'' + \alpha (k_3^2 - k_5) = 0 .$$

\noindent Let us now replace $\mathcal{c}_1''$ using \eqref{c1p}. The result becomes an ODE for $\mathcal{b}_1$: for some constants $\alpha_j$,

\begin{equation*}
\mathcal{b}^{(4)}_1 + \sum_{j=0}^3 \alpha_j \mathcal{b}^{(j)}_1 = \alpha_4 .
\end{equation*}

\noindent $\bullet$ If $k_0 \neq 0$, by rescaling let $k_0 = 1$ and by considering $e^{-k_1 z} k(z)$ instead of $k(z)$ (see Remark~\ref{REM multiplier sesqui}) we may assume $k_1 = 0$. Note that $\mathcal{c}_2 = \mathcal{c}_1 + k_2 (\mathcal{b}_1 - \mathcal{b}_2)$, using this in \eqref{nth derivative of sesqui rel} with $n=1$, we get 

\begin{equation} \label{c_1p}
\mathcal{c}_1' = -k_3 (\mathcal{b}_1 + \mathcal{b}_2) - k_2 (2\mathcal{b}_1' + \mathcal{b}_2') .
\end{equation}

\noindent The relation for $n=2$ reads

$$-k_2 (\mathcal{b}_1'' + 2 \mathcal{b}_2'') + k_3 (\mathcal{b}_1' - 3\mathcal{b}_2') + (k_4 - k_2^2) (\mathcal{b}_1 - \mathcal{b}_2) - \mathcal{c}_1'' = 0 ,$$

\noindent and replacing $\mathcal{c}_1''$ using \eqref{c_1p} we obtain

\begin{equation*}
k_2 (\mathcal{b}_1'' - \mathcal{b}_2'') + 2 k_3 (\mathcal{b}_1' - \mathcal{b}_2') + (k_4 - k_2^2) (\mathcal{b}_1 - \mathcal{b}_2) = 0 .
\end{equation*}

\noindent Consider the following cases:

\begin{enumerate}

\item If $k_2 = k_3 = 0$, then we are going to show that $k$ is trivial. Assume first that $\mathcal{b}_1 \neq \pm \mathcal{b}_2$, so from the above equation $k_4 = 0$. Further, we see that in this case $\mathcal{c}_1 = \mathcal{c}_2 = \text{const}$. Let now $k_j = 0$ for $j=1,...,n+1$, then \eqref{nth derivative of sesqui rel} reads

\begin{equation*}
k_{n+2} \left[ (-1)^n \mathcal{b}_1 - \mathcal{b}_2 \right] = 0 ,
\end{equation*}

\noindent so $k_{n+2} = 0$ and by induction $k_j = 0$ for any $j \neq 0$, i.e. $k$ is trivial. When $\mathcal{b}_1 = \mathcal{b}_2$, or $\mathcal{b}_1 = - \mathcal{b}_2$ the result follows analogously.

\item If $k_2 = 0$ and $k_3 \neq 0$, then $\mathcal{b}_2(y) = \mathcal{b}_1(y) + \alpha e^{\tau y}$ with $\tau = -\frac{k_4}{2 k_3}$ and some $\alpha \in \CC$. From \eqref{nth derivative of sesqui rel} with $n=3$ (by replacing $\mathcal{c}_1'''$ using \eqref{c_1p}) we find 

\begin{equation*}
\mathcal{c}_1 = -\tfrac{\alpha}{2k_3} (5 \tau^2 k_3 + 4 \tau k_4 + k_5) e^{\tau y} -2 \mathcal{b}_1'' - \tfrac{5k_4}{2k_3} \mathcal{b}_1' - \tfrac{k_5}{k_3} \mathcal{b}_1  .
\end{equation*}

\noindent Finally we replace this and $\mathcal{b}_2$ in \eqref{c_1p} to obtain, for some other constants $\alpha_j$

\begin{equation*}
\mathcal{b}_1^{(3)} + \sum_{j=0}^2 \alpha_j \mathcal{b}_1^{(j)} = \alpha_3 e^{\tau y} .
\end{equation*}

\item If $k_2 \neq 0$, then $\mathcal{b}_2(y) = \mathcal{b}_1(y) + f(y)$ and $f$ solves $k_2 f'' + 2k_3 f' + (k_4 - k_2^2)f = 0$, so either $f(y) = \lambda_1 e^{\tau_1 y} + \lambda_2 e^{\tau_2 y}$ or $f(y) = (\lambda_1 y + \lambda_2) e^{\tau y}$. Using the ODE for $f$, \eqref{nth derivative of sesqui rel} for $n=3$ can be written as

\begin{equation*}
4k_2 \mathcal{b}_1''' + 6k_3 \mathcal{b}_1'' + 5k_4 \mathcal{b}_1' + 2k_5 \mathcal{b}_1 + \mathcal{c}_1''' + 3k_2 \mathcal{c}_1' + 2k_3 \mathcal{c}_1 = -k_4 f' + (k_2 k_3 - k_5) f .
\end{equation*}

\noindent Let us now replace $\mathcal{c}_1'''$ and $\mathcal{c}_1'$ in the above relation using \eqref{c_1p}. The result becomes

\begin{equation*}
\begin{split}
2k_3 \mathcal{c}_1 =& -k_2 \mathcal{b}_1''' - 4k_3 \mathcal{b}_1'' + (9k_2^2 - 5k_4) \mathcal{b}_1' + (6 k_2 k_3 - 2k_5) \mathcal{b}_1 +
\\
&+ (4k_2^2 - 2k_4 + 2\tfrac{k_3^2}{k_2}) f' + (3k_2k_3 - k_5 + \tfrac{k_3 k_4}{k_2}) f ,
\end{split}
\end{equation*}

\noindent but because $f'$ has the same form as $f$ we can rewrite the above relation as

\begin{equation*}
2k_3\mathcal{c}_1(y) = - k_2 \mathcal{b}_1''' +  \sum_{j=0}^2 \gamma_j \mathcal{b}_1^{(j)}(y) + f(y) ,
\end{equation*}

\noindent with different constants $\lambda_j$ in $f$ and $\gamma_j$ are some constants. Now if $k_3 = 0$ we got an ODE for $\mathcal{b}_1$, otherwise divide by it and substitute the obtained expression and the expression of $\mathcal{b}_2$ into \eqref{c_1p}, the result is (with different constants)

\begin{equation*}
\mathcal{b}^{(4)}_1 + \sum_{j=0}^3 \gamma_j \mathcal{b}_1^{(j)} = f(y) .
\end{equation*}

\end{enumerate}


\section{Reduction of the general case} \label{reduction SECTION}

In this section we prove Theorem~\ref{THM reduction}, i.e. if $k$ is nontrivial, then $L_1 = L_2$ or $L_1 = -L_2$. Analysis of the previous section shows that $\mathcal{b}_j, \mathcal{c}_j$ are linear combinations of polynomials multiplied with an exponential, moreover the polynomials have degree at most five. So let us consider a typical such term:

\begin{equation*}
\mathcal{b}_1(y) \leftrightarrow \left( \sum_{j=0}^5 b_j y^j \right) e^{\lambda y}, \qquad \mathcal{c}_1(y) \leftrightarrow \left( \sum_{j=0}^5 c_j y^j \right) e^{\lambda y} ,
\end{equation*}

\noindent and analogous terms in $\mathcal{b}_2, \mathcal{c_2}$ only with possibly different coefficients $\tilde{b}_j, \tilde{c}_j$ respectively. Set $k(z) = \kappa(z) e^{-\frac{\lambda}{2}z}$ and let

\begin{equation} \label{kappa+-}
\kappa_+(z) = \tfrac{1}{2} [\kappa(z) + \kappa(-z)], \qquad \kappa_-(z) = \tfrac{1}{2} [\kappa(z) - \kappa(-z)] .
\end{equation}

\noindent  Substituting the expressions for $\mathcal{b}_j, \mathcal{c}_j$ and $k(z) = e^{-\frac{\lambda}{2}z} [\kappa_+(z) + \kappa_-(z)]$ into \eqref{R1}, we obtain that a linear combination of terms $y^j e^{\lambda y}$ is zero. From linear independence we conclude that each coefficient must vanish. In particular, the relation corresponding to $y^5 e^{\lambda y}$ reads

\begin{equation*}
(b_5 - \tilde{b}_5) \kappa_+'' - \left( (b_5 - \tilde{b}_5) \tfrac{\lambda^2}{4} + \tilde{c}_5 - c_5 \right) \kappa_+ - (b_5 + \tilde{b}_5) \kappa_-'' + \left( (b_5 + \tilde{b}_5) \tfrac{\lambda^2}{4} - \tilde{c}_5 - c_5 \right) \kappa_- =0 .
\end{equation*}

\noindent Because $\kappa_+$ is even, and $\kappa_-$ is odd we can add the above relation, with $z$ replaced by $-z$, to itself. Like this we separate the above relation into two ODEs one for $\kappa_+$ and the other for $\kappa_-$:

\begin{equation*}
\begin{cases}
(b_5 - \tilde{b}_5) \kappa_+'' - \left( (b_5 - \tilde{b}_5) \tfrac{\lambda^2}{4} + \tilde{c}_5 - c_5 \right) \kappa_+ = 0,
\\
(b_5 + \tilde{b}_5) \kappa_-'' - \left( (b_5 + \tilde{b}_5) \tfrac{\lambda^2}{4} - \tilde{c}_5 - c_5 \right) \kappa_- =0 .
\end{cases}
\end{equation*}

\noindent If $b_5 \neq \pm \tilde{b}_5$, then $\kappa_+ = \cosh(\mu z)$ and $\kappa_-$ is either $z$ or $\sinh(\mu z)$ for some $\mu \in \CC$, therefore $k$ is trivial. Therefore, we consider the following cases:

\noindent $\bullet$ $b_5 = \tilde{b}_5$, then obviously $c_5 = \tilde{c}_5$ and we get $b_5 \kappa_-'' - \left( \frac{b_5 \lambda^2}{4} - c_5 \right) \kappa_- = 0$. Assume $b_5 \neq 0$, then by normalization we can make $b_5 = 1$, now with $\mu^2 = \frac{\lambda^2}{4} - c_5$

\begin{equation*}
\kappa_-(z) = 
\begin{cases}
\alpha z, &\mu = 0,
\\
\alpha \sinh (\mu z), \quad &\mu\neq 0.
\end{cases}
\end{equation*}

\noindent Using the ODE that $\kappa_-$ solves, the even part of the relation corresponding to $y^4 e^{\lambda y}$ reads

\begin{equation*}
(b_4 - \tilde{b}_4) \kappa_+'' - \left( (b_4 - \tilde{b}_4) \tfrac{\lambda^2}{4} + \tilde{c}_4 - c_4 \right) \kappa_+ = 0,
\end{equation*}

\noindent which immediately implies $b_4 = \tilde{b}_4$, and hence $c_4 = \tilde{c}_4$. Odd part of that relation is

\begin{equation*}
z\kappa_+'' + 2\kappa_+' - \mu^2 z \kappa_+ = -\tfrac{2b_4}{5} \kappa_-'' + \left(\tfrac{b_4 \lambda^2}{10} - \tfrac{2 c_4 }{5} + \lambda \right) \kappa_-.
\end{equation*}  

\noindent Making the change of variables $\kappa_+(z) = \frac{u(z)}{z}$, the left-hand side of the above relation becomes $u'' - \mu^2 u$, therefore using the expression for $\kappa_-$ and the evenness of $\kappa_+$ we find

\begin{equation*}
\kappa_+(z) = 
\begin{cases}
\alpha_1 z^2 + \alpha_0, &\mu=0,
\\
\alpha_1 \cosh(\mu z) + \alpha_0 \frac{\sinh \mu z}{z}, \quad &\mu \neq 0.
\end{cases}
\end{equation*} 

\noindent If $\kappa_+$ is given by the first formulas, then $k$ is trivial. Therefore, we assume $\mu \neq 0$ and the second formula holds. The even part of the relation for $y^3 e^{\lambda y}$ is

\begin{equation*}
\begin{split}
(-10z^2 + b_3 - \tilde{b}_3) \kappa_+'' - 20z \kappa_+' + \left[ \left( \tfrac{5\lambda^2}{2} -10c_5 \right) z^2 - (b_3-\tilde{b}_3) \tfrac{\lambda^2}{4} + c_3 - \tilde{c}_3 \right] \kappa_+ =
\\
= 4b_4 z \kappa_-'' - (b_4 \lambda^2 - 4 c_4 +10\lambda) z \kappa_-.
\end{split}
\end{equation*}

\noindent When we substitute the formulas for $\kappa_{\pm}$ and multiply the relation by $z^3$, the result has the form

\begin{equation*}
p(z) e^{\mu z} - p(-z) e^{-\mu z} = 0,
\end{equation*} 

\noindent where $p(z) = \sum_{j=0}^4 p_j z^j$, therefore by linear independence we conclude that all the coefficients of $p$ vanish, in particular one can compute that $p_0 =-2 \alpha_0 (b_3 - \tilde{b}_3)$ and $p_2 = \alpha_0 \left(-(b_3 - \tilde{b}_3) \mu^2 + (b_3 - \tilde{b}_3) \tfrac{\lambda^2}{4} + \tilde{c}_3 - c_3 \right)$, if $\alpha_0 = 0$, then obviously $k$ is trivial, so $p_0 = 0$ implies $b_3 = \tilde{b}_3$, but then $p_2 = 0$ implies $c_3 = \tilde{c}_3$. Looking at the even part of the relation coming from $y^2 e^{\lambda y}$ we obtain an analogous equation, where the polynomial $p$ may be of 5th order, but expressions of $p_0, p_2$ stay the same, only the subscripts of $b_3, \tilde{b}_3, c_3, \tilde{c}_3$ change to two. And we conclude $b_2 = \tilde{b}_2$ and $c_2 = \tilde{c}_2$. Likewise looking at the even parts of the relations coming from $y e^{\lambda y}, e^{\lambda y}$ we find $b_j = \tilde{b}_j$ and $c_j = \tilde{c}_j$ for $j=1,0$.

When we look at another term with $\left( \sum_{j=0}^5 b_j' y^j \right)e^{\lambda' y}$ in the coefficient $\mathcal{b}_1$ (and similar terms for other coefficient functions) we must have $b_5' = \tilde{b}_5'$, otherwise $k$ is trivial.

If $b_5 = 0$, the same procedure applies, we only need to relabel the coefficients in the above equations. Thus our conclusion is that $L_1 = L_2$.

\noindent $\bullet$ $b_5 = - \tilde{b}_5$, this case is analogous to the previous one and the conclusion is $L_1 = -L_2$.


\section{$L_1 = L_2$}

In this section we aim to prove Theorem~\ref{THM sesqui comm}. Item 1 (in the limiting case $\gamma = 0$) and item 2 of Theorem~\ref{THM sesqui comm} are derived in Corollary~\ref{nu=1, m=2 CORRO}. Item 1 (in the case $\gamma \neq 0$) and item 3 are derived in Sections~\ref{SECT item 1}, \ref{SECT item 3}. So let us assume the setting of Theorem~\ref{THM sesqui comm}.

The analysis in the beginning of Section~\ref{Sesqui comm SECTION} shows that $\mathcal{b}$ solves a linear homogeneous ODE with constant coefficients of order at most 4. Hence $\mathcal{b}(y)$ is a linear combination of terms like $y^l e^{\lambda_j y}$, where $\lambda_j$ (called also a \textit{mode}) is a root of fourth order polynomial. We will see that there are two major cases: $\Re \lambda_j = 0$ (\textit{type 1}) or $\Re \lambda_j \neq 0$ (\textit{type 2}). In the former case $k(z)$ is given in three possible forms featuring a free real-valued and even function (cf. \eqref{k(z) type 1}). In the latter case $k(z)$ is determined and has two possible forms (cf. \eqref{k(z) type 2}). 

In Section~\ref{SECT single mode and mult} we analyze the multiplicity of the mode $\lambda_j$, in particular type 2 mode cannot have multiplicity larger than one, as is shown in Lemma~\ref{type 2 mult>1 impossible LEMMA}, while type 1 mode can have multiplicity at most 3 as established in Lemma~\ref{type 1 mult > 3 impossible LEMMA}.

Finally, in Section~\ref{SECT multiple modes} we turn to the question of analyzing possibilities of having multiple modes, i.e. distinct roots $\lambda_j$. In Corollary~\ref{CORO 3 type 1 modes impossible} we show that having three distinct type 1 modes is impossible. In Corollary~\ref{CORO 3 type 2 modes impossible} we show that having three distinct type 2 modes is impossible. In Lemma~\ref{two type 1 roots with one mult>1 LEMMA} we show that two distinct type 1 modes with one of them having multiplicity at least 2 leads to trivial kernels. And in Lemma~\ref{type 1 with mult>1 and type 2 LEMMA} we show that having type 1 mode with multiplicity at least 2 and a type 2 mode again leads to trivial kernels. So the only cases leading to nontrivial kernels are: two type 2 and one type 1 mode all with multiplicity one analyzed in Section~\ref{SECT item 1}; and two type 1 modes with multiplicity 1 analyzed in Section~\ref{SECT item 3}.

Throughout this section, until Section~\ref{SECTION L2=-L1} we will be working with $k(-z)$ and with an abuse of notation it will be denoted by $k(z)$. We will remember about this notational abuse when collecting the results in Theorem~\ref{THM sesqui comm}. In particular \eqref{R1} becomes

\begin{equation} \label{sesqui comm rel}
\mathcal{b}(y) k''(z)-\mathcal{b}(y+z)k''(-z)-\mathcal{b}'(y) k'(z)+\mathcal{b}'(y+z)k'(-z)+\mathcal{c}(y)k(z)-\mathcal{c}(y+z)k(-z)=0 .
\end{equation}

The analysis in the beginning of the Section~\ref{Sesqui comm SECTION} shows that $\mathcal{b}$ solves a linear homogeneous ODE with constant coefficients of order at most 4, and that

\begin{equation} \label{c relation when k0 not 0}
-k_0 \mathcal{c}'(y) + 2k_1 \mathcal{c}(y) + k_1 \mathcal{b}''(y) - 3k_2 \mathcal{b}'(y) + 2k_3 \mathcal{b}(y) = 0 .
\end{equation}

\noindent So $\mathcal{b}$ has the following form

\begin{equation} \label{b(y) combination of exps}
\mathcal{b}(y) = \sum_{j=1}^\nu p_{d_j} (y) e^{\lambda_j y},
\end{equation}

\noindent where $\lambda_1,...,\lambda_\nu$ are distinct complex numbers and $p_{d_j}$ are polynomials of degree $d_j$, so that

\begin{equation*}
\nu + \sum_{j=1}^\nu d_j \leq 4.
\end{equation*}

\noindent Then $\mathcal{c}(y)$ satisfying \eqref{c relation when k0 not 0}  must also have the same form, except the polynomials are different and there could be an extra exponential term $e^{\frac{2k_1}{k_0} y}$, if $\frac{2k_1}{k_0} \notin \{ \lambda_1,...,\lambda_\nu \}$. Because we also require $\mathcal{b}(\pm 1) = 0$, then either

\begin{enumerate}
\item[I.] $\nu = 1$, \ $d_1 \geq 1$;
\item[II.] $\nu = 2$, \ $d_1 \geq 1$;
\item[III.] $\nu=2, \ d_1=d_2=0, \ \mathcal{b}(y) = e^{i\beta y} \sin (\pi n (y-1)/2)$ for some $\beta \in \RR$ and $n \geq 1$;
\item[IV.] $\nu \geq 3$.
\end{enumerate}

\subsection{Single mode and multiplicities} \label{SECT single mode and mult}

In this section we concentrate on the single mode $\lambda$ and analyze its multiplicity. So suppose $p(y) e^{\lambda y}$ is one of the terms in \eqref{b(y) combination of exps}, while $q(y) e^{\lambda y}$ is one of the terms in $\mathcal{c}(y)$. Where $p(y) = \sum_{j=0}^4 p_j y^j$ and $q(y) = \sum_{j=0}^4 q_j y^j$. We are going to show that type 2 mode cannot have multiplicity larger than one (see Lemma~\ref{type 2 mult>1 impossible LEMMA}), while type 1 mode cannot have multiplicity larger than 3 (see Lemma~\ref{type 1 mult > 3 impossible LEMMA}). Finally, here we also derive item 1 (in the limiting case $\gamma = 0$) and item 2 of Theorem~\ref{THM sesqui comm} (see  Corollary~\ref{nu=1, m=2 CORRO}).

After substitution of the corresponding expressions for $\mathcal{b}, \mathcal{c}$ into \eqref{sesqui comm rel}, we collect the coefficients of $y^j e^{\lambda y}$ and from linear independence conclude that they must be zero. Like this we obtain $5$ relations involving $k$. Let us first change the variables $k(z) = \kappa(z) e^{\lambda z /2}$, then the relation corresponding to $y^j e^{\lambda y}$ can be conveniently written as

\begin{equation} \label{5 relations}
\begin{split}
p_j \kappa''(z) - \tfrac{p^{(j)}(z)}{j!}  \kappa''(-z)+ \tfrac{p^{(j+1)}(z)}{j!}  \kappa'(-z)
-(j+1) p_{j+1} \kappa'(z) +
\\
+\tfrac{\varepsilon^{(j)}(z)}{j!} \kappa(-z) - \varepsilon_j \kappa(z)&=0,\quad j=0,\ldots,4,
\end{split}
\end{equation}

\vspace{.1in}

\noindent with the convention that $p_5 = 0$, and the notation

\begin{equation*}
\varepsilon(z) = \sum_{j=0}^4 \varepsilon_j z^j,
\qquad \varepsilon_j = \tfrac{\lambda^2 p_j}{4} - q_j + \tfrac{(j+1)}{2} \lambda p_{j+1} .
\end{equation*}

 Let $\deg(p)=m$ and $\deg(q)=n$, and $\kappa_{+}, \kappa_{-}$ be the even and odd parts of $\kappa$, respectively. If $n>m$ the relation in \eqref{5 relations} for $j=n$ reads $q_n \kappa_-(z) = 0$, so $k(z) = \kappa_+(z) e^{\lambda z /2}$, the symmetry (A) implies $\lambda = 2 i \beta$ for some $\beta \in \RR$ and that $\kappa_+$ is real valued.

Let now $n \leq m$, then \eqref{5 relations} for $j=m$ reads 

\begin{equation} \label{y^m relation and mu def}
\kappa_-''(z) - \mu^2 \kappa_-(z) = 0, \qquad \mu = \sqrt{\tfrac{\lambda^2}{4} - \tfrac{q_m}{p_m}} ,
\end{equation}

\noindent hence there are two possibilities: if $\mu=0$, then $\kappa_-(z) = \alpha z + \beta$ and if $\mu \neq 0$, then $\kappa_-(z) = \alpha e^{\mu z} + \beta e^{-\mu z}$, using that $\kappa_-$ is an odd function we conclude 

\begin{equation} \label{kappa_-}
\kappa_-(z) =
\begin{cases}
\alpha z ,  &\mu=0,
\\
\alpha \sinh(\mu z), \qquad &\mu \neq 0.
\end{cases}
\end{equation}

\noindent Thus, $k(z) = e^{\lambda z /2} \left( \kappa_+(z) + \kappa_-(z) \right)$, where $\kappa_+$ is a free even function. Now the symmetry condition (A) says

\begin{equation} \label{symmetry for kappas}
e^{\overline{\lambda} z /2} \left( \overline{\kappa_+(z)} + \overline{\kappa_-(z)} \right) = e^{-\lambda z /2} \left( \kappa_+(z) - \kappa_-(z) \right).
\end{equation}

\noindent This equation can be solved uniquely for $\kappa_+$ if and only if $\Re \lambda \neq 0$. 

If $\lambda = 2 i \beta$, then $\kappa_+$ can be arbitrary real and even function, while solvability implies that

\begin{equation} \label{k(z) type 1}
k(z) = e^{i \beta z} \left( \kappa_+(z) + 
\begin{cases}
i \alpha z, & \mu = 0
\\
i \alpha \sinh (\mu z), \quad &\mu \neq 0
\\
i \alpha \sin (\mu z), \quad &\mu \neq 0
\end{cases} \right) ,
\end{equation} 

\noindent where $\alpha, \mu \in \RR$. Observe that the case $n>m$ is included here when we take $\alpha = 0$, therefore we may assume $m \geq n$.

\begin{remark} \label{mu i*mu remark}
When $\kappa_-$ is given by the second formula of \eqref{kappa_-}, then \eqref{symmetry for kappas} implies that there are two cases, either $\alpha \in i\RR$ and $\mu \in \RR$ which gives the second formula of \eqref{k(z) type 1}, or $\alpha \in \RR$ and $\mu \in i\RR$, which gives the third one, where with the abuse of notation we denoted the imaginary part of $\mu$ again by $\mu$.  
\end{remark}

If $\lambda = 2 \gamma + 2i \beta$ with $\gamma \neq 0$, then

\begin{equation} \label{k(z) type 2}
k(z) = 
\begin{cases}
\displaystyle z e^{i \beta z} \frac{\alpha e^{-\gamma z} + \overline{\alpha} e^{\gamma z}}{\sinh(2 \gamma z)}, &\mu = 0,
\\[3ex]
\displaystyle e^{i \beta z} \frac{\alpha e^{-\gamma z} \sinh(\mu z) + \overline{\alpha} e^{\gamma z} \sinh (\overline{\mu} z)}{\sinh(2 \gamma z)}, \qquad &\mu \neq 0 ,
\end{cases}
\end{equation}

\noindent where $\alpha, \mu \in \CC$.

So far we have analyzed only one of the relations from \eqref{5 relations} and deduced the possible forms of $k$. When the mode $\lambda$ has multiplicity at least two we have $m \geq 1$, and therefore there are more relations in \eqref{5 relations} that $k$ has to satisfy (in particular the one corresponding to $j=m-1$). In the two subsections below we analyze these possibilities.

\subsubsection{Type 1 mode and multiplicities}

\begin{prop} \label{type 1 mult >=2 PROP}
Let $\Re \lambda  = 0$ and $m \geq 1$, then with $\lambda = 2i \beta$ and $\alpha, \mu, \varkappa, \kappa_0 \in \RR$ we have (in fact $\varkappa = i \alpha \omega$ with $\omega$ defined in \eqref{y^m-1 relation} below)

\begin{equation} \label{k(z) type 1 m>=1}
k(z) = e^{i \beta z} \cdot 
\begin{cases}
i \alpha z + \kappa_0 + \frac{\varkappa}{6} z^2,   & \mu = 0,
\\
i \alpha \sinh (\mu z) + \kappa_0 \frac{\sinh \mu z}{z} + \frac{\varkappa}{2 \mu} \cosh \mu z, \qquad &\mu \neq 0,
\\
i \alpha \sin (\mu z) + \kappa_0 \frac{\sin \mu z}{z} - \frac{\varkappa}{2 \mu} \cos \mu z, \qquad &\mu \neq 0.
\end{cases}
\end{equation}

\end{prop}

\begin{proof}
So we see that the function $\kappa_+$ in \eqref{k(z) type 1} is not arbitrary and we are going to find it from the relation \eqref{5 relations} with $j=m-1$ (because $m \neq 0$ we can consider the index $m-1$). Recall that w.l.o.g. we assumed $m\geq n$, note that $p^{(m-1)}(z) = m! p_m z + (m-1)! p_{m-1}$, \ $\varepsilon_m = \frac{\lambda^2 p_m}{4} - q_m$ and  $\varepsilon_{m-1} = \frac{\lambda^2 p_{m-1}}{4} - q_{m-1}+\frac{m}{2} \lambda p_m$ so we obtain

\begin{equation*}
\begin{split}
p_{m-1} \kappa''(z) - (m p_m z + p_{m-1}) \kappa''(-z) + mp_m [\kappa'(-z) - \kappa'(z)] +& 
\\
+[m \varepsilon_m z + \varepsilon_{m-1}] \kappa(-z) - \varepsilon_{m-1} \kappa(z) &= 0 .
\end{split}
\end{equation*}

\noindent Now using \eqref{y^m relation and mu def} we can rewrite the above relation as

\begin{equation} \label{y^m-1 relation}
z \kappa_+'' + 2\kappa_+' - \mu^2 z \kappa_+ = \omega \kappa_- ,
\qquad \quad
\omega = -\lambda + \tfrac{2}{m p_m} \left( q_{m-1} - \tfrac{q_m p_{m-1}}{p_m} \right) ,
\end{equation}

\noindent where $\kappa_-$ appears in the three formulas from \eqref{k(z) type 1}. 

According to Remark~\ref{mu i*mu remark}, when $\kappa_-(z) = i \alpha \sin \mu z$, in the above relation $\mu$ should be replaced by $i\mu$, which changes the sign of the last term on LHS from negative to positive. This explains the difference of the sign in the second and third formulas of \eqref{k(z) type 1 m>=1}. Solving the obtained ODE, recalling that $\kappa_+$ is even and real valued, we find \eqref{k(z) type 1 m>=1} with $\varkappa = i \alpha \omega$.

\end{proof}

When $m \geq 2$, we can consider \eqref{5 relations} with $j=m-2$, moreover we know that \eqref{y^m relation and mu def} and \eqref{y^m-1 relation} also hold, and using these and $p^{(m-2)}(z) = \frac{m!}{2} p_m z^2 + (m-1)! p_{m-1} z + (m-2)! p_{m-2}$, the relation with $j=m-2$ can be simplified to

\begin{equation} \label{y^m-2 relation}
z \kappa_-' + \eta_1 \kappa_- = \eta_2 z \kappa_+, \qquad \quad \eta_2 = \frac{\omega}{2} ,
\end{equation} 

\noindent where $\omega$ is defined in \eqref{y^m-1 relation} and $\eta_1$ is a constant whose precise expression is not important.

\begin{prop}
Let $\Re \lambda = 0$ and $m \geq 2$, then with $\lambda = 2 i \beta$ and $\alpha, \kappa_0, \mu \in \RR$

\begin{equation} \label{k(z) for type 1 and m>=2}
k(z) = e^{i \beta z} \cdot
\begin{cases}
\displaystyle \kappa_0 \frac{\sinh \mu z}{z},
\vspace{.1in}
\\
\displaystyle \alpha e^{i \mu z} + \kappa_0 \frac{\sin \mu z}{z}.
\end{cases}
\end{equation}

\noindent Moreover, in the second case the following relations between the involved parameters must be satisfied

\begin{equation} \label{karevor}
\kappa_0 \eta_2 = i \alpha \eta_1 , \qquad \qquad \eta_2 = \pm i \mu .
\end{equation}
\end{prop}

\begin{proof}

By Proposition~\ref{type 1 mult >=2 PROP} we know what are the functions $\kappa_-$ and $\kappa_+$ that satisfy the two relations \eqref{5 relations} with $j=m, m-1$ (they are given in the three formulas in \eqref{k(z) type 1 m>=1}, with $\varkappa = i \alpha \omega$). Here we want to see which of these satisfy the third relation \eqref{y^m-2 relation}. First note that $\varkappa \in \RR$ implies $\omega$ and hence also $\eta_2 = \frac{\omega}{2}$ are purely imaginary. The case \eqref{k(z) type 1 m>=1}a implies that $k$ has rank at most three and so, is trivial.

If \eqref{k(z) type 1 m>=1}b holds, then \eqref{y^m-2 relation} after multiplying by $2 \mu$ reads

\begin{equation*}
z (2i\alpha \mu^2 - \eta_2 \varkappa) \cosh (\mu z) + 2 \mu (i\alpha \eta_1 - \eta_2 \kappa_0) \sinh(\mu z) = 0.
\end{equation*}

\noindent By linear independence we conclude that the two coefficients must vanish: $2i\alpha \mu^2 - \eta_2 \varkappa = 0$ and $i \alpha \eta_1 - \eta_2 \kappa_0 = 0$. Let us ignore the second equation (it just gives some restrictions on $q_j$'s), using the expression for $\varkappa$ the first one becomes $\alpha (\mu^2 - \eta_2^2) = 0$. If $\alpha \neq 0$, because $\eta_2 \in i \RR$, we conclude $\mu = \eta_2 = 0$ which is a contradiction. Thus $\alpha = 0$, which gives the first formula of \eqref{k(z) for type 1 and m>=2}.

If \eqref{k(z) type 1 m>=1}c holds, then \eqref{y^m-2 relation} reads 

\begin{equation*}
z (2i\alpha \mu^2 + \eta_2 \varkappa) \cos (\mu z) + 2 \mu (i\alpha \eta_1 - \eta_2 \kappa_0) \sin(\mu z) = 0 .
\end{equation*}

\noindent Again the two coefficients must be zero, the second one implies the first relation of \eqref{karevor} and the first one gives $\alpha (\mu^2 + \eta_2^2) = 0$. One possibility is $\alpha = 0$, another one: when $\alpha \neq 0$, then $\Im \eta_2 = \pm \mu$, hence we may write $\kappa (z) = \pm \alpha (\cos \mu z \pm i \sin \mu z ) + \kappa_0 \frac{\sin \mu z}{z} = \pm \alpha e^{\pm i \mu z}+ \kappa_0 \frac{\sin \mu z}{z}$. These cases can be unified in the second formula of \eqref{k(z) for type 1 and m>=2}.  
\end{proof}

\begin{corollary} \label{nu=1, m=2 CORRO}
When there is one type 1 root with multiplicity three (i.e. $\nu = 1$, $m=2$ and $\lambda = 2i \beta$), we obtain item 1 (in the limiting case $\gamma = 0$) and item 2 of Theorem~\ref{THM sesqui comm}.

\end{corollary}

\begin{proof}
Using the boundary conditions $\mathcal{b}(y) = (y^2-1) e^{\lambda y}$, we know $k$ from the above proposition so it only remains to find $\mathcal{c}$. Before that let us invoke Remark~\ref{REM multiplier sesqui} and w.l.o.g. assume that $\beta = 0$, or equivalently $\lambda = 0$.

 From \eqref{c relation when k0 not 0} we know that $\mathcal{c}(y) = \sum_{j=0}^3 c_j y^j + c_4 e^{\tau y}$ with $\tau \neq 0$. Clearly $\mu \neq 0$, otherwise $k$ is trivial (see \eqref{k(z) for type 1 and m>=2}). We substitute these expressions into \eqref{sesqui comm rel} and obtain that a linear combination of $e^{\tau y}$ and monomials $y^j$ is zero, hence by linear independence each of the coefficients must vanish. The equation coming from the term $e^{\tau y}$ reads
 
\begin{equation} \label{e^tau}
c_4 \left[ k(z) - e^{\tau z} k(-z) \right] = 0 .
\end{equation}

\noindent Equations coming from the terms $y^3,...,1$, respectively are

\begin{equation} \label{y^3..1}
\begin{split}
c_3 \left[ k(z) - k(-z) \right] &= 0,
\\
k''(z) - k''(-z) + c_2 k(z) - (3c_3 z + c_2) k(-z) &= 0,
\\
2z k''(-z) + 2 k'(z) - 2 k'(-z) - c_1 k(z) + (3 c_3 z^2 + 2 c_2 z + c_1) k(-z) &= 0,
\\
k''(z) + (z^2-1)k''(-z) - 2z k'(-z) - c_0 k(z) + (c_3 z^3 + c_2 z^2 + c_1 z + c_0) k(-z) &= 0 .
\end{split}
\end{equation}

Assume $k$ is given by the first formula of \eqref{k(z) for type 1 and m>=2}, in particular it is even and \eqref{e^tau} implies $c_4 = 0$. The first equation of \eqref{y^3..1} is identity, the second one implies $c_3 \sinh(\mu z) = 0$ and hence $c_3=0$. Third one reads $(c_2 + \mu^2) \sinh (\mu z) = 0$, hence $c_2 = - \mu ^2$. Finally, the fourth relation simplifies to $c_1 \sinh (\mu z) = 0$, so that $c_1 = 0$. We note that $c_0$ remains free. Thus, we conclude that $\mathcal{c}(y) = -\mu^2 y^2 + c_0$ and since we are free to choose $c_0$, we can rewrite $\mathcal{c}$ as $\mathcal{c}(y) = -\mu^2 \mathcal{b}(y) + c_0$, which proves item 1 of Theorem~\ref{THM sesqui comm} in the case $\gamma = 0$ and $\mu \in \RR$. 

Assume $k$ is given by the second formula of \eqref{k(z) for type 1 and m>=2}. Because $\kappa_0 \neq 0$, we may normalize it to be one. \eqref{e^tau} reads

$$c_4 \left[ e^{-i\mu z} - e^{i\mu z}  + e^{(i\mu+\tau) z} - e^{(-i\mu+\tau) z} + i\alpha z (e^{(-i\mu+\tau) z} - e^{i\mu z}) \right] = 0 ,$$

\noindent and from the linear independence of the involved exponentials we get $c_4 = 0$. The first equation of \eqref{y^3..1} reads $c_3 \alpha \sin (\mu z) = 0$, and there are two cases to consider. 

If $\alpha = 0$, the second equation reads $c_3 \sin (\mu z) = 0$, so $c_3 = 0$. The third equation becomes  $(c_2 - \mu^2) \sin (\mu z) = 0$, hence $c_2 = \mu^2$. Finally, the fourth equation implies $c_1 = 0$ and again $c_0$ is free. So we find $\mathcal{c}(y) = \mu^2 \mathcal{b}(y) + c_0$, which proves item 1 of Theorem~\ref{THM sesqui comm} in the case $\gamma = 0$ and $\mu \in i\RR$. 

If $\alpha \neq 0$, then $c_3 = 0$. The second equation of \eqref{y^3..1} implies $c_2 = \mu^2$, the third one: $c_1 = 2i\mu$ and finally the fourth one implies $c_0 = - \mu^2 + \frac{2\mu}{\alpha}$. Thus, $\mathcal{c}(y) = \mu^2 (y^2-1) + 2 i \mu y + \frac{2 \mu}{\alpha}$, which proves item 2 of Theorem~\ref{THM sesqui comm}.  

\end{proof}

\begin{lemma} \label{type 1 mult > 3 impossible LEMMA}
Let $\Re \lambda = 0$ and $m \geq  3$, then $k$ is trivial.
\end{lemma}

\begin{proof}

By the previous proposition we know that $\kappa(z)$ has two possible forms coming from \eqref{k(z) for type 1 and m>=2}. The goal is to show that it cannot solve \eqref{5 relations} with $j = m-3$. Using the equations \eqref{y^m relation and mu def}, \eqref{y^m-1 relation} and \eqref{y^m-2 relation} we can rewrite the relation for $j=m-3$ as 

\begin{equation} \label{y^m-3 relation}
(\eta_2 z^2 + \eta_3) \kappa_- = z^2 \kappa_+' +  3 \eta_1 z \kappa_+ ,
\end{equation}

\noindent where $\eta_1, \eta_2$ are the same as in \eqref{y^m-2 relation} and the expression for $\eta_3$ is not important. 

When $k$ is given by the first formula of \eqref{k(z) for type 1 and m>=2}, $\kappa_-(z) = 0$ and $\kappa_+(z) = \kappa_0 \frac{\sinh(\mu z)}{z}$ so \eqref{y^m-3 relation} implies $\mu = 0$ and hence $k=0$.

When $k$ is given by the second formula of \eqref{k(z) for type 1 and m>=2}, let us w.l.o.g. take $\kappa_0 = 1$. As we saw in the previous proposition $\kappa_-(z) = i \alpha \sin (\mu z)$ and $\kappa_+(z) = \frac{\sin(\mu z)}{z} - \frac{i \alpha \eta_2}{\mu} \cos (\mu z)$ with $\eta_2 = \pm i \mu$ and $i \alpha \eta_1= \eta_2$. Let first $\eta_2 = i \mu$, then substituting $\kappa_\pm$ into \eqref{y^m-3 relation} we get

\begin{equation*}
\left[ i \alpha \eta_3 + \kappa_0 (1 - 3 \eta_1) \right] \sin (\mu z) - z \left( \mu  + 3 \alpha \eta_1 \right) \cos (\mu  z) = 0 .
\end{equation*}

\noindent But then $\mu  + 3 \alpha \eta_1 = 4 \mu$ which must be zero, hence $k$ is trivial. The case $\eta_2 = - i \mu$ is done analogously.

\end{proof}

\subsubsection{Type 2 mode and multiplicities}

\begin{lemma} \label{type 2 mult>1 impossible LEMMA}
Let $\Re \lambda \neq 0$ and $m\geq 1$, then $k = 0$.
\end{lemma}

\begin{proof}

Let $\lambda = \gamma + i\beta$, with $\gamma \neq 0$, \eqref{symmetry for kappas} implies

\begin{equation*}
\begin{cases}
\kappa_+ - \overline{\kappa}_+ e^{\gamma z} =  \overline{\kappa}_- e^{\gamma z} + \kappa_-,
\\
\overline{\kappa}_+  -  \kappa_+ e^{\gamma z} =  \kappa_-  e^{\gamma z} + \overline{\kappa}_-,
\end{cases}
\end{equation*}

\noindent where the second equation was obtained by conjugating the first one, then

\begin{equation} \label{kappa_+ in terms of kappa_-}
\kappa_+ = - \coth (\gamma z) \kappa_- - \csch (\gamma z) \overline{\kappa}_- .
\end{equation}

\noindent We know that both of the relations \eqref{y^m relation and mu def} and \eqref{y^m-1 relation} hold. When $\mu = 0$, we have $\kappa_-(z) = \varkappa z$, hence $\kappa_+(z) = \frac{\omega \alpha}{6} z^2 + \kappa_0$ and comparing this with \eqref{kappa_+ in terms of kappa_-} we conclude $k=0$. So let us assume $\mu \neq 0$, then from \eqref{kappa_-}, $\kappa_-(z) = \alpha \sinh(\mu z)$, hence solving the ODE \eqref{y^m-1 relation} we get

\begin{equation*}
\kappa_+(z) = c_2 \frac{\sinh (\mu z)}{z} + \frac{\varkappa \alpha}{2\mu} \cosh(\mu z) ,
\end{equation*}

\noindent substitute this into \eqref{kappa_+ in terms of kappa_-} divide the result by $\sinh (\mu z)$ to get

\begin{equation*}
\frac{c_2}{z} + \frac{\varkappa \alpha}{2\mu} \coth(\mu z) = -\alpha \coth (\gamma z) - \overline{\alpha} \frac{\sinh (\overline{\mu} z)}{\sinh (\mu z)} \csch (\gamma z) .
\end{equation*}

\noindent Assume $\gamma > 0$ (otherwise negate $(\gamma, \alpha, \varkappa)$), write $\mu = \mu_1 + i\mu_2$, assume $\mu_1 \neq 0$, then we may assume $\mu_1 > 0$, otherwise multiply the equation by $-1$. Now consider the asymptotics as $z \to +\infty$,

\begin{equation*}
\frac{c_2}{z} + \frac{\varkappa \alpha}{2\mu} = -\alpha - 2\overline{\alpha} e^{-\gamma z} e^{-2i\mu_2 z} ,
\end{equation*}

\noindent clearly this implies $\alpha=c_2=0$, so $k=0$. Let now $\mu_1 = 0$, then the relation reads

\begin{equation*}
\frac{c_2}{z} - \frac{\varkappa \alpha}{2\mu_2} \cot(\mu_2 z) = -\alpha \coth (\gamma z) + \overline{\alpha} \csch (\gamma z) ,
\end{equation*}

\noindent and asymptotics at $+\infty$ gives $\frac{c_2}{z} - \frac{\varkappa \alpha}{2\mu_2} \cot(\mu_2 z) = -\alpha  + 2 \overline{\alpha} e^{-\gamma z}$ which again implies $\alpha=c_2=0$.

\end{proof}

\subsection{Multiple modes} \label{SECT multiple modes}

Before we start to analyze the possibilities of having multiple distinct modes $\lambda_j$ in \eqref{b(y) combination of exps}, we state that in view of Lemmas~\ref{type 1 mult > 3 impossible LEMMA} and ~\ref{type 2 mult>1 impossible LEMMA} the cases I and II can be rewritten

\begin{enumerate}
\item[I.] $\nu=1, \ d_1=2, \ \Re \lambda_1=0$;
\item[IIa.] $\nu=2, \ d_1 \geq 1, \ \Re \lambda_1 = \Re \lambda_2 =0$;
\item[IIb.] $\nu=2, \ d_1 \geq 1, \ \Re \lambda_1 =0, \  \Re \lambda_2 \neq 0$.
\end{enumerate}

\noindent The case I was analyzed in Corollary~\ref{nu=1, m=2 CORRO}, so it remains to consider cases IIa,b and III, IV. We will see in Lemmas~\ref{two type 1 roots with one mult>1 LEMMA} and ~\ref{type 1 with mult>1 and type 2 LEMMA} that the cases IIa,b lead to trivial kernels $k$. Case III will be analyzed in Section~\ref{SECT item 3}. We will show that case IV is only possible when there are exactly three modes: two type 1 and one type 2, all with multiplicity one. This case will then be analyzed in Section~\ref{SECT item 1}.  

When $\lambda_j = 2i \beta_j$ (of course $\beta_1 \neq \beta_2$) then \eqref{k(z) type 1} holds true for both of the modes $\lambda_j$ and we determine the free functions and conclude

\begin{equation} \label{k(z) two type 1 roots}
k(z) = \frac{\alpha_1 k_s(\mu_1 z) e^{i\beta_1 z} + \alpha_2 k_r(\mu_2 z) e^{i\beta_2 z}}{\sin (\beta_1-\beta_2) z}, \qquad r,s \in \{1,2,3\},
\end{equation}

\noindent where all the constants are real, $\mu_j \neq 0$ and $k_r$ is given by

\begin{equation} \label{k1, k2 and k3}
k_1(t)=t, \quad k_2(t) = \sin t,  \quad k_3(t) = \sinh t.
\end{equation}

\begin{prop} \label{beta1 and 2 are determined by k PROP}
Let $k$ be given by \eqref{k(z) two type 1 roots}, then $\beta_1$ and $\beta_2$ are determined by $k$.
\end{prop}

\begin{proof}

W.l.o.g. let $\beta_1-\beta_2 > 0$, otherwise swap $
\beta_1$ with $\beta_2$; $r$ with $s$; $\mu_1$ with $\mu_2$ and replace $(\alpha_1, \alpha_2)$ by $(-\alpha_2, -\alpha_1)$. There are six cases to consider.

\vspace{.1in}

\noindent $\bullet$ If $(s,r) = (3,3)$; we have 
\[
k(it) = e^{-\beta_1 t} \cdot \frac{\alpha_1 \sin (\mu_1 t) + \alpha_2 \sin
  (\mu_2 t) e^{(\beta_1-\beta_2)t}}{\sinh (\beta_1-\beta_2)t },
\] 
therefore
\begin{equation*}
k(it) \sim 
\begin{cases}
  2\alpha_1 \sin (\mu_1 t) e^{(\beta_2 - 2\beta_1)t} + 2\alpha_2 e^{-\beta_1 t} \sin (\mu_2 t),
&t \to +\infty,\\
2\alpha_1 \sin (\mu_1 t) e^{-\beta_2 t} + 2\alpha_2 e^{(\beta_1 - 2\beta_2)t}
\sin (\mu_2 t),  &t \to -\infty .
\end{cases}
\end{equation*}

\noindent When $(s,r) = (1,1)$ the same formulas hold with $\sin (\mu_j t)$ replaced by $t$ for $j=1,2$. And when $(s,r) = (1,3)$ the same formulas hold with $\sin (\mu_1 t)$ replaced by $t$. The above asymptotics immediately conclude the proof in this case.

\noindent $\bullet$ If  $(s,r) = (2,3)$, we may assume $\mu_1 > 0$, otherwise negate $\alpha_1$, so
\[
k(it) = e^{-\beta_1 t} \cdot \frac{\alpha_1 \sinh (\mu_1 t) + \alpha_2 \sin
  (\mu_2 t) e^{(\beta_1-\beta_2)t}}{\sinh (\beta_1-\beta_2)t },
\] 
and therefore

\begin{equation*}
k(it)\sim  
\begin{cases}
    \alpha_1 e^{(\mu_1 + \beta_2 - 2\beta_1)t} + 2\alpha_2 \sin (\mu_2 t) e^{-\beta_1 t} , 
&t \to +\infty,\\
\alpha_1 e^{-(\mu_1+\beta_2) t} + 2 \alpha_2 \sin (\mu_2 t) e^{(\beta_1 - 2\beta_2)t}, &t \to -\infty.
  \end{cases}
\end{equation*}

\noindent If $\alpha_2 \neq 0$ clearly $\beta_1$ and $\beta_2$ are determined. So assume $\alpha_2 = 0$, then from the above asymptotics we conclude that $\alpha_1$, $\mu_1 + \beta_2$ and $\beta_1$ are determined. But note that $k_0:=k(0) = \frac{\mu_1 \alpha_1}{\beta_1 - \beta_2}$, so we have a system ($k_1$ denotes a parameter determined by $k$)

\begin{equation*}
\begin{cases}
\alpha_1 \mu_1 + k_0 \beta_2 = k_0 \beta_1
\\
\mu_1 + \beta_2 = k_1
\end{cases}
\end{equation*}

\noindent Which is not solvable w.r.t. $\mu_1$ and $\beta_2$ \IFF $k_0 = \alpha_1$, but in this case the first equation implies $\beta_1 - \beta_2 = \mu_1$, therefore $k(z) = \alpha_1 e^{i \beta_1 z}$ which is trivial. When $(s,r) = (2,1)$ the asymptotic formulas hold with $\sin (\mu_2 t)$ replaced by $t$ and the same argument applies.

\noindent $\bullet$ If  $(s,r) = (2,2)$, we may assume $\mu_1,\mu_2 > 0$, otherwise negate $\alpha_1,\alpha_2$, so
\[
k(it) = e^{-\beta_1 t} \cdot \frac{\alpha_1 \sinh (\mu_1 t) + \alpha_2 \sinh
  (\mu_2 t) e^{(\beta_1-\beta_2)t}}{\sinh (\beta_1-\beta_2)t },
\] 
therefore
\begin{equation*}
k(it)\sim
\begin{cases}
  \alpha_1 e^{(\mu_1 + \beta_2 - 2\beta_1)t} + \alpha_2 e^{(\mu_2-\beta_1) t}, 
 &t \to +\infty, \\
\alpha_1 e^{-(\mu_1+\beta_2) t} + \alpha_2 e^{-(\mu_2-\beta_1 + 2\beta_2)t},&t \to -\infty.
\end{cases}
\end{equation*}

\noindent If $\alpha_1, \alpha_2 \neq 0$, clearly $\beta_1$ and $\beta_2$ are determined. Assume $\alpha_1 = 0$, then from the above asymptotics we conclude that $\alpha_2, \mu_2 - \beta_1$ and $\beta_2$ are determined. Next, as above we look at $k(0) = \frac{\mu_2 \alpha_2}{\beta_1 - \beta_2}$, and conclude that $\beta_1, \mu_2$ are not determined \IFF $\mu_2 = \beta_1 - \beta_2$ in which case $k$ is trivial. Analogous conclusion holds in the case $\alpha_2=0$.

\end{proof}

\begin{corollary} \label{CORO 3 type 1 modes impossible}
Having three distinct modes $\lambda_1, \lambda_2, \lambda_3 \in i \RR$ is impossible.
\end{corollary}

\begin{lemma} \label{two type 1 roots with one mult>1 LEMMA}
Having two distinct type 1 modes, one of them with multiplicity at least two leads to a trivial kernel. In other words, if $k(z)$ can be written in the form \eqref{k(z) type 1 m>=1} and \eqref{k(z) two type 1 roots}, then $k$ is trivial.
\end{lemma}

\begin{proof}
The denominator in \eqref{k(z) two type 1 roots} is zero when $z=\pi n / (\beta_1 - \beta_2)$. If the numerator does not vanish at all of these values then the function in \eqref{k(z) two type 1 roots} is not entire, while all functions \eqref{k(z) type 1 m>=1} are entire. Thus it must hold

\begin{equation*}
\alpha_1 k_s \left(\tfrac{\pi \mu_1 n}{\beta_1 - \beta_2} \right) + (-1)^n \alpha_2 k_r \left(\tfrac{\pi \mu_2 n}{\beta_1 - \beta_2} \right) = 0 \qquad \forall n \in \ZZ.
\end{equation*} 

\noindent This equation can hold in three cases $(r,s) = (2,2), (2,3)$ or $(1,2)$. Let us consider the first one, the other two can be analyzed similarly, and in fact are simpler. The solutions of the above equation for $r=s=2$ are

\begin{enumerate}
\item[(a)] $\mu_j = m_j (\beta_1 - \beta_2)$ with $m_j \in \ZZ$ for $j=1,2$,

\item[(b)] $\alpha_1 = \pm \alpha_2$ , \ $\mu_2 = (2m_1+1) (\beta_1 - \beta_2) \mp \mu_1$.

\end{enumerate}

\noindent In both of these cases $k$ is a trigonometric polynomial. But if $k$ is given by \eqref{k(z) type 1 m>=1} and is a trigonometric polynomial, then $k(z) = e^{i \beta z} (i \alpha \sin \mu z + \alpha' \cos \mu z)$ for some constants $\alpha, \alpha', \beta$ and $\mu$. Showing that $k$ is trivial.

\end{proof}

\begin{lemma} \label{type 2 determines constants LEMMA}
Let $k$ be given by \eqref{k(z) type 2}, then the pair $(|\gamma|, \beta)$ is determined by $k$. 
\end{lemma}

\begin{proof}

Let $k$ be given by the first formula, assume $\gamma>0$, otherwise replace $(\gamma, \alpha)$ with $(-\gamma, - \overline{\alpha})$, then

\begin{equation} \label{type 2a asymp}
k(z) \sim 2 \overline{\alpha} z e^{-\gamma z} e^{i \beta z}, \qquad \text{as} \ z \to +\infty ,
\end{equation} 

\noindent so $\alpha, \gamma, \beta$ are determined by $k$. But note that the sign of $\gamma$ is not determined.

Let now $k$ be given by the second formula, write $\mu = \mu_1 + i \mu_2$ and $\alpha = \alpha_1 + i\alpha_2$,

\begin{enumerate}

\item let $\mu_1 \neq 0$, we may assume $\mu_1 > 0$, otherwise we replace $(\alpha, \mu)$ with $(-\alpha,-\mu)$. Also assume $\gamma > 0$, otherwise we replace $(\gamma, \alpha, \mu)$ with $(-\gamma, -\overline{\alpha},\overline{\mu})$, then

\begin{equation} \label{type 2b asymp mu1 not 0}
k(z) \sim \overline{\alpha} e^{(-\gamma + \mu_1)z} e^{i(\beta - \mu_2)z}, \qquad \text{as} \ z \to +\infty ,
\end{equation} 

\noindent so $\alpha, -\gamma + \mu_1$ and $\beta - \mu_2$ are determined by $k$. We then note that $k(0) = \frac{\Re (\alpha \mu)}{\gamma}$ and $k'(0) = i \beta k(0) -  i \Im (\alpha \mu)$. Because of the symmetry of $k$, we know that $k(0) \in \RR$ and $k'(0) \in i \RR$, so let us set $k_0 = k(0)$ and $k_1 = \frac{k'(0)}{i}$, then we obtain the system

\begin{equation*}
\begin{cases}
\alpha_1 \mu_1 - \alpha_2 \mu_2 - k_0 \gamma = 0,
\\
-\alpha_2 \mu_1 - \alpha_1 \mu_2 + k_0 \beta = k_1,
\\
\mu_1 - \gamma = k_2,
\\
-\mu_2 + \beta = k_3,
\end{cases}
\qquad A = 
\begin{pmatrix}
\alpha_1 && -\alpha_2 && -k_0 && 0 \\
-\alpha_2 && -\alpha_1 && 0 && k_0 \\
1 && 0 && -1 && 0 \\
0 && -1 && 0 && 1
\end{pmatrix},
\end{equation*}

\noindent where the unknowns are $\mu_1, \mu_2, \gamma, \beta$ and $k_2, k_3$ are parameters determined by $k$. The system is linear and one can compute $\det (A) = (\alpha_1 - k_0)^2 + \alpha_2^2$. If $\det(A) \neq 0$, then the system has a unique solution and all the constants $\mu_1, \mu_2, \gamma, \beta$ are determined by the function $k$. Of course we see that the signs of $\gamma$ and $\mu_1$ are not determined. 

When $\det(A) = 0$, we get $\alpha_1 = k_0$ and $\alpha_2 = 0$, then (note that $k_0 \neq 0$, because otherwise $k=0$). Now we must have $k_2 = 0$ and $k_3 = \frac{k_1}{k_0}$ and the above system reduces to

\begin{equation*}
\begin{cases}
\mu_1 - \gamma = 0,
\\
-\mu_2 + \beta = k_3.
\end{cases}
\end{equation*}

\noindent So $\alpha$ is real and $\mu_1 = \gamma$, and in this case one can check that the formula reduces to $k(z) = \alpha e^{i (\beta + \mu_2)z}$ which is a trivial kernel.

\item $\mu_1 = 0$, we may assume $\gamma > 0$, otherwise replace $(\gamma, \alpha)$ by $(-\gamma, \overline{\alpha})$, then

\begin{equation} \label{type 2b asymp mu1 0}
k(z) \sim \overline{\alpha} e^{-\gamma z} \left[ e^{i(\beta - \mu_2)z} - e^{i(\beta + \mu_2)z} \right] \qquad \text{as} \ z \to + \infty ,
\end{equation}

so $\alpha, \gamma, \beta, \mu_2$ are determined by $k$. And again we see that the sign of $\gamma$ is not determined. 

\end{enumerate}

\end{proof}

\begin{corollary} \label{two type 2 modes CORO}
Let $\lambda_j = 2 \gamma_j + i 2 \beta_j$, with $\gamma_j \neq 0$ for $j=1,2$. Assume $\lambda_1 \neq \lambda_2$, then $\lambda_2 = - \overline{\lambda_1}$.
\end{corollary}

\begin{proof}
For each $\lambda_j$, $k$ can be given by two formulas from \eqref{k(z) type 2}, let us refer to them as "a" and "b". There are three cases to consider: (a,a); (b,b) and (a,b). By comparing the asymptotics \eqref{type 2b asymp mu1 not 0} and \eqref{type 2b asymp mu1 0} with \eqref{type 2a asymp} we see that they cannot be matched, hence the third case is impossible. Consider the first one, then

\begin{equation*}
k(z) = z e^{i \beta_j z} \cdot \frac{\alpha_j e^{-\gamma_j z} + \overline{\alpha_j} e^{\gamma_j z}}{\sinh (2\gamma_j z)}, \qquad j=1,2.
\end{equation*}

\noindent As we saw $|\gamma_j|$ and $\beta_j$ are determined by $k$, hence we conclude $|\gamma_1| = |\gamma_2|$ and $\beta_1 = \beta_2$. Because $\lambda_1 \neq \lambda_2$ we must have $\gamma_1 = - \gamma_2$. The second case is done analogously.

\end{proof}

\begin{corollary} \label{CORO 3 type 2 modes impossible}
Having three distinct modes $\lambda_1, \lambda_2, \lambda_3 \notin i \RR$ leads to trivial $k$..
\end{corollary}

\begin{lemma} \label{type 1 with mult>1 and type 2 LEMMA}
Having a type 2 mode and a type 1 mode of multiplicity at least two leads to a trivial kernel. In other words, if $k(z)$ can be written in the form \eqref{k(z) type 1 m>=1} and \eqref{k(z) type 2}, then $k$ is trivial.
\end{lemma}

\begin{proof}

So $\lambda_1 = i 2 \beta_1$ and $\lambda_2 = 2\gamma + i2 \beta_2$ with $\gamma \neq 0$.
All the functions in \eqref{k(z) type 1 m>=1} are entire, and one can easily check that the first function of \eqref{k(z) type 2} is entire \IFF $\alpha=0$, which leads to $k=0$. So let us consider the case when $k$ is given by the second formula:

\begin{equation} \label{(1^2, 2) main equation}
k(z) =
e^{i \beta_2 z} \cdot \frac{\alpha_2 e^{- \gamma z} \sinh(\mu z) + \overline{\alpha_2} e^{\gamma z} \sinh(\overline{\mu} z)}{\sinh (2 \gamma z)} =e^{i \beta_1 z}
\begin{cases}
i\alpha_1 z + \kappa_0 + \frac{\varkappa}{6} z^2,
\\
i \alpha_1 \sinh \mu_0 z + \kappa_0 \frac{\sinh \mu_0 z}{z} + \frac{\varkappa}{2\mu_0} \cosh \mu_0 z,
\\
i \alpha_1 \sin \mu_0 z + \kappa_0 \frac{\sin \mu_0 z}{z} - \frac{\varkappa}{2\mu_0} \cos \mu_0 z,
\end{cases}
\end{equation}

\noindent where $\mu_0 (\neq 0), \alpha_1, \kappa_0, \varkappa \in \RR$, and write $\mu = \mu_1 + i \mu_2$. 

\vspace{.1in}

\noindent \textbf{Case 1}: if $\mu_1 \neq 0$, may assume $\mu_1 > 0$ and $\gamma > 0$. If $k$ is given by the

\begin{enumerate}

\item 1st formula, then comparing the asymptotics we see that $\alpha_1 = \varkappa = 0$, then for the LHS $k(z) \sim \kappa_0 e^{i \beta_1 z}$. Again comparing we find $\overline{\alpha_2} = \kappa_0$, $-\gamma + \mu_1 = 0$ and $\beta_2 - \mu_2 = \beta_1$. The last two conditions can be rewritten as $\lambda_2 - \lambda_1 = 2 \mu$, and so $k(z) = \kappa_0 e^{i \beta_1 z}$, which is trivial.

\item 2nd formula, we may assume $\mu_0 >0$, otherwise negate $(\alpha_1, \kappa_0, \varkappa)$, then 

$k(z) \sim \frac{1}{2} (i\alpha_1 + \frac{\varkappa}{2 \mu_0}) e^{\mu_0 z} e^{i\beta_1 z}$, comparing with \eqref{type 2b asymp mu1 not 0} we conclude

\begin{equation*}
-\gamma + \mu_1 = \mu_0, \qquad \beta_2 - \mu_2 = \beta_1, \qquad i\alpha_1 + \tfrac{\varkappa}{2 \mu_0} = 2 \overline{\alpha_2} ,
\end{equation*} 

with these, in \eqref{(1^2, 2) main equation} we express $\sinh$ and $\cosh$ in terms of exponentials, by linear independence we conclude that $\kappa_0=0$, and obtain

\begin{equation*}
- \overline{\alpha_2} e^{(\gamma - \mu_1) z} + \alpha_2 e^{(\gamma - \mu_1)z} = e^{i2\mu_2 z} \left[ \alpha_2 e^{(-3\gamma + \mu_1)z} - \overline{\alpha_2} e^{(-3\gamma - \mu_1) z} \right] .
\end{equation*} 

Hence $\mu_2 = 0$, then using that $\gamma, \mu_1 \neq 0$ we deduce that the above relation is possible (with $\alpha_2 \neq 0$) \IFF $\mu_1 = 2\gamma$. Thus $k(z) = e^{i\beta_1 z} \left[ i \alpha_1 \sinh \mu_0 z +  \tfrac{\varkappa}{2\mu_0} \cosh \mu_0 z \right]$ is trivial.

\item 3rd formula, we may assume $\mu_0 >0$, otherwise negate $(\alpha_1, \kappa_0, \varkappa)$, then 

$k(z) \sim e^{i\beta_1 z} \left[ ( \frac{\alpha_1}{2} - \frac{\varkappa}{4 \mu_0}) e^{i\mu_0 z} - ( \frac{\alpha_1}{2} + \frac{\varkappa}{4 \mu_0}) e^{-i\mu_0 z} \right] $, comparing this with \eqref{type 2b asymp mu1 not 0} we conclude $-\gamma + \mu_1 = 0$ and

\begin{enumerate}
\item $\beta_1 + \mu_0 = \beta_2-\mu_2$, \ $\frac{\alpha_1}{2} - \frac{\varkappa}{4 \mu_0} = \overline{\alpha_2}$ and $\frac{\alpha_1}{2} + \frac{\varkappa}{4 \mu_0} = 0$ , or

\item $\beta_1 - \mu_0 = \beta_2-\mu_2$, \ $\frac{\alpha_1}{2} - \frac{\varkappa}{4 \mu_0} = 0$ and $\frac{\alpha_1}{2} + \frac{\varkappa}{4 \mu_0} = -\overline{\alpha_2}$
\end{enumerate}

Let us consider the first option, in that case \eqref{(1^2, 2) main equation} simplifies to $\kappa_0 e^{i \beta_1 z} \frac{\sin \mu_0 z}{z} = 0$ which implies $\kappa_0 = 0$, and so $k(z) = \alpha_1 e^{i(\beta_1 + \mu_0)}$. The other case is done analogously.

\end{enumerate}

\vspace{.1in}

\noindent \textbf{Case 2}: if $\mu_1 = 0$, we may assume $\gamma > 0$. If $k$ is given by the 1st or 3rd formulas, comparing the asymptotics of LHS with \eqref{type 2b asymp mu1 0} we conclude $\gamma = 0$, which is a contradiction, so these cases lead to $k=0$. Now let $k$ be given by the second formula, again w.l.o.g let $\mu_0 > 0$, then we see that the asymptotics cannot be matched because in \eqref{type 2b asymp mu1 0} $e^{i(\beta_2 \pm \mu_2)z}$ are linearly independent, hence $k=0$. 

\end{proof}

\begin{lemma} \label{type 1 and 2 LEMMA}
Let $\lambda_1 = i 2 \beta_1$ and $\lambda_2 = 2\gamma + i2 \beta_2$, with $\gamma \neq 0$, then $\beta_1 = \beta_2 =: \beta$ and 

\begin{equation} \label{k(z) type 1 and 2}
k(z) = \alpha e^{i \beta z} \frac{k_r(\mu z)}{\sinh \gamma z}, \quad r \in \{1,2,3\},
\end{equation}

\noindent where $\alpha, \mu \in \RR$ and $k_r$ is defined in \eqref{k1, k2 and k3}.

\end{lemma}

\begin{proof}
So $k$ is given by both of the forms \eqref{k(z) type 2} and \eqref{k(z) type 1}. Assume $k$ is given by the first formula of \eqref{k(z) type 2}, then we can find

\begin{equation*}
\kappa_+(z) = z e^{i \Delta \beta z} \frac{\alpha e^{-\gamma z} + \overline{\alpha} e^{\gamma z}}{\sinh(2\gamma z)} - i\alpha' k_r(\mu' z), \quad r \in \{1,2,3\} ,
\end{equation*}

\noindent where $\Delta \beta = \beta_2-\beta_1$, $ 0 \neq \mu', \alpha' \in \RR$. It is easy to check that $\kappa_+$ as above satisfies $\kappa_+(-z) = \overline{\kappa_+(z)}$, hence $\kappa_+$ is real valued \IFF it is even, and with $\alpha = \alpha_1 + i \alpha_2$ the imaginary part of $\kappa_+$ being zero reads

\begin{equation} \label{(1,2) with 2a}
z\alpha_1 \frac{\sin (\Delta \beta z)} {\sinh(\gamma z)} - z\alpha_2 \frac{\cos (\Delta \beta z)} {\cosh(\gamma z)} = \alpha' k_r(\mu' z) . 
\end{equation} 

\noindent We may assume $\gamma > 0$, otherwise replace $(\gamma, \alpha_1)$ with $(-\gamma, -\alpha_1)$. Assume $k \neq 0$, note that

\begin{equation*}
\text{LHS} \sim 2 z e^{- \gamma z} [\alpha_1 \sin (\Delta \beta z) - \alpha_2 \cos (\Delta \beta z)], \qquad \text{as} \ z \to + \infty.
\end{equation*}

\noindent Comparing this with the asymptotic of RHS for $r=1,2,3$ we conclude that \eqref{(1,2) with 2a} is possible \IFF $\Delta \beta = 0$ and $\alpha_2 = \alpha' =0$. And we see that $k$ is given by \eqref{k(z) type 1 and 2} with $r=1$.

Assume now $k$ is given by the second formula of \eqref{k(z) type 2}, then

\begin{equation*}
\kappa_+(z) = e^{i \Delta \beta z} \cdot \frac{\alpha e^{- \gamma z} \sinh(\mu z) + \overline{\alpha} e^{\gamma z} \sinh(\overline{\mu} z)}{\sinh (2 \gamma z)} - i\alpha' k_r(\mu' z), \qquad r \in \{1,2,3\} .
\end{equation*}

\noindent Write $\mu= \mu_1 + i \mu_2$ and $\alpha = \alpha_1 + i \alpha_2$, w.l.o.g. let $\gamma > 0$, assume $\mu_1 \neq 0$ then we can assume $\mu_1 > 0$; again $\kappa_+$ being even and real valued are equivalent and $\Im \kappa_+ = 0$ reads

\begin{equation} \label{(1,2) with 2b}
\begin{split}
&\frac{\sin (\Delta \beta z)} {\sinh(\gamma z)} [\alpha_1 \sinh (\mu_1 z) \cos (\mu_2 z) - \alpha_2 \cosh (\mu_1 z) \sin (\mu_2 z)] -
\\
- &\frac{\cos (\Delta \beta z)} {\cosh(\gamma z)} [\alpha_1 \cosh (\mu_1 z) \sin (\mu_2 z) + \alpha_2 \sinh (\mu_1 z) \cos (\mu_2 z)] = \alpha' k_r(\mu' z).
\end{split}
\end{equation} 

\noindent We note that as $z \to \infty$

\begin{equation*}
\text{LHS} \sim e^{(-\gamma + \mu_1) z} \left[ \alpha_1 \sin (\Delta \beta - \mu_2) z - \alpha_2 \cos (\Delta \beta - \mu_2)z \right] ,
\end{equation*}

\noindent comparing this with the asymptotic of RHS for r=1,2,3 we conclude that \eqref{(1,2) with 2b} is possible for non-trivial $k$ \IFF $\Delta \beta = \mu_2$ and $\alpha_2 = \alpha' =0$. (For example when $r=2$, \eqref{(1,2) with 2b} is also possible when $\mu_1 = \gamma$, $\alpha_2=0$, $\alpha'=\alpha_1$ and $\Delta \beta - \mu_2 = \mu'$ but in this case one easily checks that $k$ is trivial). Now \eqref{(1,2) with 2b} reduces to

\begin{equation*}
\sin (2 \mu_2 z) \left[ \frac{\sinh \mu_1 z}{\sinh \gamma z} - \frac{\cosh \mu_1 z}{\cosh \gamma z} \right] = 0 .
\end{equation*}

\noindent If the second factor is zero, we must have $\gamma = \mu_1$ and in this case $k$ reduces to a trivial kernel. So $\mu_2 = 0$, and $k$ is given by \eqref{k(z) type 1 and 2} with $r=3$. Let now $\mu_1 = 0$, then \eqref{(1,2) with 2b} becomes

\begin{equation} \label{(1,2) with 2b and mu1=0}
- \sin (\mu_2 z) \left[ \alpha_2 \frac{\sin \Delta \beta z}{\sinh \gamma z} + \alpha_1 \frac{\cos \Delta \beta z}{\cosh \gamma z} \right] = \alpha' k_r(\mu' z) .
\end{equation}

\noindent We note that as $z \to \infty$

\begin{equation*}
\text{LHS} \sim -2 e^{-\gamma z} \sin (\mu_2 z) \left[ \alpha_2 \sin (\Delta \beta  z) + \alpha_1 \cos (\Delta \beta z) \right] ,
\end{equation*}

\noindent comparing this with the asymptotics of RHS for r=1,2,3 we find that \eqref{(1,2) with 2b and mu1=0} is possible for non-trivial $k$ \IFF $\Delta \beta = 0$ and $\alpha_1 = \alpha' =0$. And $k$ is given by \eqref{k(z) type 1 and 2} with $r=2$.

\end{proof}

\begin{corollary}
Having three distinct modes $\lambda_1, \lambda_2 \in i\RR$ and $\lambda_3 \notin i\RR$ is impossible.
\end{corollary}

\subsection{Item 1, $\gamma \neq 0$} \label{SECT item 1}

The previous analysis shows that case IV is only possible when we have exactly three modes $\lambda_1, \lambda_2 \notin i\RR$ and $\lambda_3 \in i\RR$ with multiplicities 1, that is $d_j=0$ for $j=1,2,3$. Moreover, by Corollary~\ref{two type 2 modes CORO} and Lemma~\ref{type 1 and 2 LEMMA} we conclude that $\lambda_1 = 2 \gamma + 2i\beta , \ \lambda_2 = -2\gamma + 2i\beta , \ \lambda_3 = 2i\beta$ and $k(z)$ is given by \eqref{k(z) type 1 and 2}. Invoking Remark~\ref{REM multiplier sesqui} let us w.l.o.g. assume $\beta = 0$. Thus,

\begin{equation*}
\lambda_1 = 2 \gamma ,\quad \lambda_2 = -2\gamma, \quad \lambda_3 = 0,
\qquad \text{and} \quad k(z) = \frac{k_r(\mu z)}{\sinh \gamma z}, \quad r \in \{1,2,3\} ,
\end{equation*} 

\noindent where $k_r$ is defined in \eqref{k1, k2 and k3}, moreover $\mathcal{b}(y) = \cosh(2 \gamma y) - \cosh(2 \gamma)$. Because of \eqref{c relation when k0 not 0}, $\mathcal{c}$ has the following form 

\begin{equation*}
\mathcal{c}(y)=(c_1 y + d_1) e^{\lambda_1 y} + (c_2 y + d_2) e^{\lambda_2 y} + (c_3y+d_3) e^{\lambda_3 y} + c_4 e^{\tau y} ,
\end{equation*}

\noindent where $\tau$ is different from all $\lambda_j$'s. Substituting these expressions into \eqref{sesqui comm rel} and looking at linearly independent parts it is easy to conclude that $c_1 = c_2 = c_3 = c_4 = 0$, and $d_1 = \frac{\lambda_1^2 + 4 \mu^2}{8}, \ d_2 = \frac{\lambda_2^2 + 4 \mu^2}{8}$ if in the formula for $k$ we have $r=2$. When $r=3$ in the expressions of $d_1, d_2$; $\mu$ should be replaced by $i \mu$ and when $r=1$, in those formulas $\mu=0$. This concludes item 1 of Theorem~\ref{THM sesqui comm} in the case $\gamma \neq 0$.

\subsection{Item 3} \label{SECT item 3}

Finally we consider the case III, because of the boundary conditions one can find that $\lambda_2 - \lambda_1 = i\pi n$ with $0 \neq n \in \ZZ$, therefore $\lambda_1, \lambda_2 \in i\RR$ (otherwise by Corollary~\ref{two type 2 modes CORO} and Lemma~\ref{type 1 and 2 LEMMA} the difference $\lambda_2 - \lambda_1$ is real). Let us now take $\lambda_1 = 2i (\beta + \tfrac{\pi n}{4})$ and $\lambda_2 = 2i (\beta - \tfrac{\pi n}{4})$ with some $\beta \in \RR$. In this case we find $\mathcal{b}(y) = e^{2i \beta y} \sin \left( \tfrac{\pi n (y-1)}{2} \right)$ and by \eqref{k(z) two type 1 roots} 

\begin{equation} \label{k item 2 sesqui}
k(z) = e^{i \beta z} \frac{\alpha_1 k_s(\mu_1 z) e^{i\pi n z/4} + \alpha_2 k_r(\mu_2 z) e^{-i\pi n z/4}}{\sin (\pi n z /2)}, \quad r,s \in \{1,2,3\}.
\end{equation}

\noindent From \eqref{c relation when k0 not 0}, $\mathcal{c}$ has the form

\begin{equation*}
\mathcal{c}(y) = (c_1 y + d_1) e^{\lambda_1 y} + (c_2 y + d_2) e^{\lambda_2 y} + c_3 e^{\tau y} ,
\end{equation*}

\noindent with $\tau \neq \lambda_j$, note that also $\tau = \frac{2 k'(0)}{k(0)} \in i \RR$. The denominator of $k$ has zeros at $z = \frac{2 m}{n}$ for $m \in \ZZ$, since we want $k$ to be smooth in $[-2,2]$, we need 

\begin{equation} \label{zeros condition}
(-1)^m \alpha_1 k_s \left( \tfrac{2\mu_1 m}{n} \right) + \alpha_2 k_r \left( \tfrac{2\mu_2 m}{n} \right) = 0, \qquad \forall m \in \ZZ \quad \text{s.t.} \ \tfrac{m}{n} \in [-1,1].
\end{equation}
\begin{enumerate}
\item $r=s=3$, if $n \neq \pm 1$, then \eqref{zeros condition} must hold for $m=1,2$, one can easily see that this leads to a contradiction. Therefore $n = \pm 1$, in which case \eqref{zeros condition} implies $\alpha_1 \sinh(2 \mu_1) = \alpha_2 \sinh(2\mu_2)$. To find $\mathcal{c}$, we substitute these expressions into \eqref{sesqui comm rel} and look at the coefficients of linearly independent parts, which must vanish. In particular the coefficient of $e^{\tau y}$ gives

\begin{equation*}
c_3 \left\{ \alpha_2 \sinh(\mu_2 z) \left[ e^{- \frac{\lambda_2 - 2\tau}{2} z} - e^{\frac{\lambda_2}{2} z} \right] +
\alpha_1 \sinh(\mu_1 z) \left[ e^{- \frac{\lambda_1 - 2\tau}{2} z} - e^{\frac{\lambda_1}{2} z} \right] \right\} = 0 .
\end{equation*}

\noindent The four exponentials in square brackets are linearly independent, moreover their exponents are purely imaginary, while $\mu_1, \mu_2$ are real, hence all the terms are linearly independent, therefore our conclusion is that $c_3 = 0$, otherwise $k=0$. Using similar arguments and looking at coefficients of $y e^{\lambda_j y}, e^{\lambda_j y}$ we find $c_1 = c_2 = 0$ and   

\begin{equation} \label{d1 and d2}
d_1 = -\frac{i e^{- \frac{i \pi}{2}}}{8} [\lambda_1^2- 4\mu_2^2], \qquad d_2 = \frac{i e^{\frac{i \pi}{2}}}{8} [\lambda_2^2- 4\mu_1^2]
\end{equation}

\item $s=1, r=3$, we can absorb $\mu_1$ into $\alpha_1$ and relabel $\mu_2$ by $\mu$, as in 1 we see $n = \pm 1$ and $2 \alpha_1 = \alpha_2 \sinh(2\mu)$. Then one can find $c_1 =c_2=c_3=0$ and \eqref{d1 and d2} holds with $\mu_2 = 0$ and $\mu_1 = \mu$. 

\item $r=s=1$, absorb $\mu_j$ into $\alpha_j$, again $n = \pm 1$ and $\alpha_1 = \alpha_2$, in which case (up to a real multiplicative constant) $k(z) = e^{i \beta z} \frac{z}{\sin (\pi z / 4)}$, then we can conclude $c_1 = c_2 = 0$, $\tau = 2i \beta$ and \eqref{d1 and d2} holds with $\mu_1= \mu_2 = 0$.

\item $s=1, r=2$, absorb $\mu_1$ into $\alpha_1$. If $n = \pm 1$ we get $2 \alpha_1 = \alpha_2 \sin (2 \mu_2)$, and following the strategy described in 1 we find $c_1 = c_2 = c_3 = 0$, and \eqref{d1 and d2} holds with $\mu_1 = 0$ and $\mu_2$ replaced by $i \mu_2$. If $|n|>1$, then \eqref{zeros condition} holds for at least $m=1,2$. It is easy to see that these two equations imply $\alpha_1 = 0$ and $\sin \left( \frac{2 \mu_2}{n} \right) = 0$. But in that case \eqref{zeros condition} holds for any $m \in \ZZ$. So $\mu_2 = \frac{\pi n l}{2} $ for some $l \in \ZZ$, hence we see that $k$ is a trigonometric polynomial, and therefore is trivial.

\item $s=3, r=2$, again if $|n| > 1$ we get $\alpha_1 = 0$ and $\sin \left( \frac{2 \mu_2}{n} \right) = 0$, which again implies $k$ is trivial. So $n = \pm 1$, and we find $\alpha_1 \sinh(2 \mu_1) = \alpha_2 \sin (2 \mu_2)$ 

\item $s=r=2$, as we saw in Lemma~\ref{two type 1 roots with one mult>1 LEMMA} if $n \neq \pm 1$, then $k$ is trivial. So $n = \pm 1$ and $\alpha_1 \sin (2 \mu_1) = \alpha_2 \sin (2 \mu_2)$, one of $\alpha_j$ is nonzero, assume it is $\alpha_2$. When $\sin(2 \mu_1) = 0$, then $\sin(2 \mu_2) = 0$ and again $k$ is a trigonometric polynomial. So $\sin (2 \mu_1) \neq 0$ and also $\sin (2 \mu_2) \neq 0$, again because of the same reason. We then find $c_1 = c_2 = 0$, \eqref{d1 and d2} holds with $\mu_j$ replaced by $i \mu_j$ for $j=1,2$. Finally the relation for $e^{\tau y}$ reads

\begin{equation*}
c_3 \left\{ \tilde{\alpha}_1 \sin(\mu_1 z) \left[ e^{(\tau - \frac{\lambda_1}{2}) z} - e^{\frac{\lambda_1}{2} z} \right] + \tilde{\alpha}_2 \sin(\mu_2 z) \left[ e^{(\tau - \frac{\lambda_2}{2}) z} - e^{\frac{\lambda_2}{2} z} \right] \right\} = 0 ,
\end{equation*}

\noindent where $\tilde{\alpha}_j = \sin (2 \mu_j) \neq 0$, $\lambda_1 - \lambda_2 = \frac{i \pi}{2}$. Now $c_3 = 0$ or the function in curly brackets (denote it by $f(z)$) vanishes, looking at the asymptotics $f(iz) $ as $z \to \infty$, and also at $f'(0), f''(0), f^{(4)}(0)$ we can find that $f =0$ \IFF $\mu_2 = \mu_1 \pm \frac{\pi}{2}$ (which implies $\tilde{\alpha}_1 = - \tilde{\alpha}_2$) and $\tau = 2i ( \beta - \frac{\pi}{4} \pm \mu_1)$.

Choosing $\beta = 0$ (cf. Remark~\ref{REM multiplier sesqui}) we conclude item 3 of Theorem~\ref{THM sesqui comm}.

\end{enumerate}


\section{$L_2 = - L_1$} \label{SECTION L2=-L1}

Assume the setting of Theorem~\ref{THM sesqui comm L2=-L1}, recall that $\mathcal{b}:=\mathcal{b}_1$ and $\mathcal{c}:=\mathcal{c}_1$. Now \eqref{R1} reads

\begin{equation} \label{L2=-L1}
\begin{split}
\mathcal{b}(y) k''(-z) + \mathcal{b}(y+z) k''(z) + \mathcal{b}'(y) k'(-z) + \mathcal{b}'(y+z) k'(z)+&
\\
+\mathcal{c}(y) k(-z) + \mathcal{c}(y+z) k(z) &= 0 .
\end{split}
\end{equation}

\noindent The analysis in the beginning of Section~\ref{Sesqui comm SECTION} shows that (in the case $L_2 = - L_1$) $\mathcal{b}(y)$ solves second order, linear homogeneous ODE with constant coefficients, and because of the boundary conditions it must be of the form

\begin{equation*}
\mathcal{b}(y) = b_1 e^{\lambda_1 y} + b_2 e^{\lambda_2 y},\quad
\mathcal{c}(y) = c_1 e^{\lambda_1 y} + c_2 e^{\lambda_2 y} + c_0 ,
\qquad \lambda_1 \neq \lambda_2,
\end{equation*}

\noindent where $\mathcal{c}$ is of the same form as $\mathcal{b}$ because it satisfies $\mathcal{c}' = - \frac{k_1}{k_0} \mathcal{b}' - \frac{k_2}{k_0} \mathcal{b}$. Clearly both $b_j$ are different from zero, and from boundary conditions

\begin{equation} \label{lambda - lambda'}
\lambda_1 - \lambda_2 = \pi i n, \qquad n \in \ZZ.
\end{equation} 

With these formulas, \eqref{L2=-L1} becomes a linear combination of functions $e^{\lambda_j y}$ with coefficients depending on $z$, hence each coefficient must vanish. Let us concentrate on the coefficient of $e^{\lambda_1 y}$, making the change of variables $k(z) = \kappa(z) e^{-\lambda_1 z / 2}$ we rewrite it as

\begin{equation*}
\kappa''_+(z) - \mu^2 \kappa_+(z) = 0, \qquad \mu = \sqrt{\tfrac{\lambda_1^2}{4} - \tfrac{c_1}{b_1}} ,
\end{equation*}  

\noindent where $\kappa_+$ is the even part of $\kappa$, because it is an even function we get

\begin{equation*}
\kappa_+(z) = \alpha \cosh (\mu z) .
\end{equation*}

\noindent The symmetry of $k$ implies

\begin{equation*}
e^{-\overline{\lambda_1} z / 2} \left( \overline{\kappa_+(z)} + \overline{\kappa_-(z)} \right) =  e^{\lambda_1 z / 2} \left( \kappa_+(z) - \kappa_-(z) \right) .
\end{equation*}

If $\lambda_1 = 2 i \beta$ with $\beta \in \RR$, then $\kappa_-$ is an arbitrary odd and purely imaginary function. Moreover, $\kappa_+$ must be real valued, hence 

\begin{equation} \label{k type 1 L2=-L1}
k(z) = e^{-i\beta z} \left(
\kappa_- (z) +
\begin{cases}
\alpha \cosh (\mu z)
\\
\alpha \cos (\mu z)
\end{cases}
\right) ,
\end{equation}

\noindent where $\alpha, \mu \in \RR$.

If $\lambda_1 = 2 \gamma + 2i \beta$ with $\gamma \neq 0$, then (recalling that $k$ is smooth at $0$), with $\kappa_0 \in \RR$

\begin{equation*}
k(z) = \alpha e^{-i \beta z} \frac{e^{\gamma z} \cosh(\mu z) -  e^{-\gamma z} \cosh(\overline{\mu} z)}{\sinh(2 \gamma z)} .
\end{equation*}

 Now $k$ should come from two distinct modes $\lambda_1, \lambda_2$, and from \eqref{lambda - lambda'} we see that $\Re \lambda_1 = \Re \lambda_2=: 2 \gamma$, so if $\gamma \neq 0$ we must have

\begin{equation*}
\alpha_1 e^{-i \beta_1 z} \left( e^{\gamma z} \cosh(\mu z) -  e^{-\gamma z} \cosh(\overline{\mu} z) \right) = \alpha_2 e^{-i \beta_2 z} \left( e^{\gamma z} \cosh(\nu z) -  e^{-\gamma z} \cosh(\overline{\nu} z) \right) ,
\end{equation*}

\noindent which implies $\beta_1 = \beta_2$, leading to a contradiction. Indeed, the function on LHS (denoted by $f(z)$) determines $\beta_1$, because with $\mu = \mu_1 + i \mu_2$

\begin{equation*}
f(iz) = \kappa_0 e^{\beta_1 z} \left[ i e^{\mu_2 z} \sin \left((\gamma - \mu_1)z \right) +  e^{-\mu_2 z} \cos \left( (\gamma + \mu_1)z \right) \right] .
\end{equation*}

\noindent Assume $\mu_2 > 0$, then $f(iz) \sim \kappa_0 e^{(\beta_1 + \mu_2) z} \sin \left((\gamma - \mu_1)z \right) $ as $z \to + \infty$, hence $\beta_1 + \mu_2$ is determined by $f$, but by looking at the asymptotics as $z \to - \infty$ we see that also $\beta_1 - \mu_2$ is determined, hence so is $\beta_1$. The case $\mu_2 \leq 0$ is done analogously.

Thus $\lambda_j = 2i \beta_j \in i\RR$ and $k$ is given by \eqref{k type 1 L2=-L1}, then $\kappa_-$ is determined and we can find

\begin{equation} \label{k when l2=-l1}
k(z) = \frac{\alpha_1 k_s'(\mu_1 z) e^{i \beta_1 z} + \alpha_2 k_r'(\mu_2 z) e^{i \beta_2 z}}{i \sin(\beta_1 - \beta_2)z}, \qquad r,s \in \{1,2,3\} ,
\end{equation}

\noindent where all the constants are real, and $k_r'$ is the derivative of function $k_r$ defined in \eqref{k1, k2 and k3}. Moreover because $k$ is smooth at $0$, we must have $\alpha_2 = - \alpha_1$. The denominator of the above function vanishes at $z = \frac{2m}{n}$ with $m \in \ZZ$, since $k$ is smooth in $[-2,2]$ we should require

\begin{equation*}
(-1)^m k_s' \left( \tfrac{2 \mu_1 m}{n} \right) - k_2' \left( \tfrac{2 \mu_2 m}{n} \right) = 0, \qquad \forall m \in \ZZ, \ \text{s.t.} \ \tfrac{m}{n} \in [-1,1].
\end{equation*} 

\noindent Because $n \neq 0$, this condition should hold at least for $m=1$. One can easily check that this implies that the functions given by \eqref{k when l2=-l1} are either zero, or trigonometric polynomials, and therefore: trivial.

\medskip

\noindent\textbf{Acknowledgments.}
This material is based upon work supported by the National Science Foundation under Grant No. DMS-1714287.

\end{document}